\documentclass[12pt]{amsart}

\usepackage[text={400pt,660pt},centering]{geometry}

\usepackage{color}
\usepackage{esint,amssymb}
\usepackage{graphicx}
\usepackage{MnSymbol}
\usepackage{mathtools}
\usepackage[colorlinks=true, pdfstartview=FitV, linkcolor=blue, citecolor=blue, urlcolor=blue,pagebackref=false]{hyperref}
\usepackage{microtype}

\usepackage{bm}
\usepackage{scalerel} 
\usepackage{dsfont}

\definecolor{darkgreen}{rgb}{0,0.5,0}
\definecolor{darkblue}{rgb}{0,0,0.7}
\definecolor{darkred}{rgb}{0.9,0.1,0.1}

\newtheorem{theorem}{Theorem}
\newtheorem{proposition}[theorem]{Proposition}
\newtheorem{lemma}[theorem]{Lemma}
\newtheorem{corollary}[theorem]{Corollary}

\theoremstyle{definition}
\newtheorem{remark}[theorem]{Remark}

\newcommand{\cref}[1]{Corollary~\ref{c.#1}}

\numberwithin{equation}{section}
\numberwithin{theorem}{section}

\newcommand{\Z}{\mathbb{Z}}
\newcommand{\N}{\mathbb{N}}
\newcommand{\R}{\mathbb{R}}

\newcommand{\A}{\mathcal{A}}
\newcommand{\Ahom}{\overline{\mathcal{A}}}

\newcommand{\innerply}[3]{\left\langle #1,#2 \right\rangle_{#3}}

\newcommand{\E}{\mathbb{E}}
\renewcommand{\P}{\mathbb{P}}
\newcommand{\F}{\mathcal{F}}
\newcommand{\Zd}{\mathbb{Z}^d}
\newcommand{\Rd}{\mathbb{R}^d}

\newcommand{\ep}{\varepsilon}
\newcommand{\eps}{\varepsilon}

\newcommand{\J}{\mathcal{J}}
\renewcommand{\a}{\mathbf{a}}
\newcommand{\ahom}{{\overbracket[1pt][-1pt]{\a}}}  

\renewcommand{\subset}{\subseteq}

\newcommand{\cu}{{\scaleobj{1.2}{\square}}}

\renewcommand{\fint}{\strokedint}

\newcommand{\Ll}{\left}
\newcommand{\Rr}{\right}

\DeclareMathOperator{\dist}{dist}

\DeclareMathOperator{\var}{var}
\DeclareMathOperator{\cov}{cov}

\newcommand{\X}{\mathcal{X}}

\renewcommand{\tilde}{\widetilde}

\newcommand{\indc}{\mathds{1}}

\newcommand{\mcl}{\mathcal}
\newcommand{\tmu}{\widetilde{\mu}}
\newcommand{\tnu}{\widetilde{\nu}}

\newcommand{\al}{\alpha}
\newcommand{\CC}{\mathbf{C}}
\newcommand{\nH}{{\underline H}}
\newcommand{\nW}{{\underline W}}

\begin{document}

\title[Mesoscopic regularity in elliptic homogenization]{Mesoscopic higher regularity and subadditivity in elliptic homogenization}

\begin{abstract}
We introduce a new method for obtaining quantitative results in stochastic homogenization for linear elliptic equations in divergence form. Unlike previous works on the topic, our method does not use concentration inequalities (such as Poincar\'e or logarithmic Sobolev inequalities in the probability space) and relies instead on a higher ($C^{k}$, $k \geq 1$) regularity theory for solutions of the heterogeneous equation, which is valid on length scales larger than a certain specified mesoscopic scale. This regularity theory, which is of independent interest, allows us to, in effect, localize the dependence of the solutions on the coefficients and thereby accelerate the rate of convergence of the expected energy of the cell problem by a bootstrap argument. The fluctuations of the energy are then tightly controlled using subadditivity. The convergence of the energy gives control of the scaling of the spatial averages of gradients and fluxes (that is, it quantifies the weak convergence of these quantities) which yields, by a new ``multiscale" Poincar\'e inequality, quantitative estimates on the sublinearity of the corrector. 
\end{abstract}

\author[S. Armstrong]{Scott Armstrong}
\address[S. Armstrong]{Ceremade (UMR CNRS 7534), Universit\'e Paris-Dauphine, Paris, France}
\email{armstrong@ceremade.dauphine.fr}

\author[T. Kuusi]{Tuomo Kuusi}
\address[T. Kuusi]{Department of Mathematics and Systems Analysis, Aalto University, Finland}
 \email{tuomo.kuusi@aalto.fi}

\author[J.-C. Mourrat]{Jean-Christophe Mourrat}
\address[J.-C. Mourrat]{Ecole normale sup\'erieure de Lyon, CNRS, Lyon, France}
\email{jean-christophe.mourrat@ens-lyon.fr}

\keywords{stochastic homogenization, higher regularity, subadditivity, sublinear corrector estimate}
\subjclass[2010]{35B27, 35B45}
\date{\today}

\maketitle

\section{Introduction}
In this paper, we introduce a new method for obtaining quantitative results in elliptic homogenization. It is based on a new regularity theory for higher derivatives of solutions which is valid on \emph{mesoscopic} scales. The regularity estimates essentially localize the dependence of the solutions on the coefficients, allowing for the mixing assumptions to dictate the rate of homogenization. This idea is formalized by a novel bootstrap argument which uses the regularity estimates to accelerate the convergence of the natural subadditive and superadditive quantities associated to the variational formulation of the equation. As an application, we give explicit estimates on the sublinearity of the correctors and the weak convergence of their rescaled gradients.

\subsection{Motivation and informal summary of results}
We consider the linear elliptic equation
\begin{equation}
\label{e.pde}
-\nabla \cdot \left( \a(x) \nabla u \right) = 0
\end{equation}
in bounded open subsets of $\Rd$. The coefficient $\a(\cdot)$ is a random field valued in the set of~$d$-by-$d$ symmetric matrices with eigenvalues belonging to the interval~$\left[ 1,\Lambda\right]$ for a fixed ellipticity constant~$\Lambda \geq 1$. The law of $\a(\cdot)$ is given by a probability measure~$\P$ which is assumed to satisfy a quantitative mixing condition.  

\smallskip

We are interested in obtaining quantitative information on the statistical properties of solutions of~\eqref{e.pde} on very large scales, a concern which lies within the realm of homogenization. The qualitative theory of homogenization for such equations was developed by~\cite{PV1,K1,Y1,JKO} and the first quantitative results are due Yurinskii~\cite{Y101}. Developing a quantitative theory of stochastic homogenization for linear elliptic equations has received a lot of attention since the groundbreaking work of Gloria and Otto~\cite{GO1,GO2,GO3} and Gloria, Neukamm and Otto~\cite{GNO2,GNO}, who proved an array of optimal estimates under certain very strong mixing assumptions on the coefficients (in particular, that~$\P$ satisfies some form of a \emph{spectral gap inequality}). At the core of their results are the use of sensitivity estimates and the spectral gap assumption to derive moment bounds on the gradient of the corrector and the Green's functions, which culminated in the work of Marahrens and Otto~\cite{MO}. These gradient estimates, following Naddaf and Spencer~\cite{NS}, give control of the sensitivity of the correctors themselves to changes in the coefficient field, and thereby yield optimal estimates after another application of the spectral gap inequality. 

\smallskip

A regularity theory for stochastic homogenization was recently introduced by Armstrong and Smart~\cite{AS}. In particular,~\cite[Theorem 1.2]{AS} provides a gradient bound, of exactly the sort required for the quantitative theory of homogenization, which applies to arbitrary solutions of~\eqref{e.pde} and with much stronger (and essentially optimal) stochastic integrability. The techniques of~\cite{AS} were further developed by Armstrong and Mourrat~\cite{AM}, for general (nonlinear) divergence-form equations and systems, and, in a similar spirit by Gloria, Neukamm and Otto~\cite{GNO3} for linear equations and systems. These works were inspired by the celebrated papers of Avellaneda and Lin~\cite{AL1,AL2}, who proved in particular uniform Lipschitz estimates for equations with periodic coefficients. As in~\cite{AL1,AL2}, the philosophy is to show that solutions of the heterogeneous equation~\eqref{e.pde} inherit the regularity of the limiting constant-coefficient equation due to homogenization. The method of~\cite{AS} is however different from these previous works, which were based on compactness arguments and used strong bounds on the correctors. Instead, the idea is to mimic more closely the proof of the classical Schauder estimates, using quantitative methods to replace compactness: one shows that, at each scale, a solution of~\eqref{e.pde} may be approximated by the solution of the homogenized equation and then iterates the resulting estimate over dyadic scales. If the error in the approximation is small enough (an algebraic or Dini rate of homogenization is needed), then the argument yields a uniform Lipschitz estimate. This argument is more robust than the compactness argument of~\cite{AL1,AL2}, since it separates the approximation step from the iteration step. Indeed, in addition to its applicability in the stochastic setting, it has yielded new results even in the periodic and almost periodic settings (cf.~\cite{ASh,Sh}). 

\smallskip

Another idea from~\cite{AS,AM} is that subadditive arguments are the way to prove quantitative homogenization results that are optimal in stochastic integrability (under essentially any mixing assumption). This leads to a regularity theory which is optimal in terms stochastic integrability. Subadditive arguments are so effective for this purpose because they reduce questions of stochastic integrability for random variables with complicated dependence on the random environment (such as solutions of the PDE) to estimates of a finite sum of bounded random variables which satisfy the same mixing properties as the coefficients. Under assumptions on the coefficient field which provide a spectral gap inequality, the regularity theory (specifically the Lipschitz estimate) is strong enough to quickly recover the optimal quantitative estimates on the corrector proved by Gloria and Otto. Thus an important consequence of these new regularity estimates on the quantitative theory was to separate the gradient estimates, which can now be proved under very general mixing assumptions, from the rest of the quantitative theory, which until now requires spectral gap-type assumptions. 

\smallskip

In the present paper, we propose to take this program one step further by deriving quantitative bounds for the correctors (with explicit exponents) without using the spectral gap or other Poincar\'e-type concentration inequalities. This is the first step in obtaining a quantitative theory of stochastic homogenization applicable to general coefficient fields. Instead of relying on concentration of measure, we show by a bootstrap argument that the regularity theory itself can be used to improve the rate of homogenization. In addition to providing a new point of view in the quantitative theory of stochastic homogenization and allowing for more general mixing assumptions, our arguments yield estimates which are much stronger in stochastic integrability (under strong mixing assumptions, we get exponential moments rather than just $p$th moments). This is because, as in \cite{AS,AM}, our method allows us to use subadditivity to control the stochastic fluctuations. So far, the estimates we can obtain (see Theorem~\ref{t.correctors}, below) are unfortunately suboptimal in their scalings compared to what can be proved under spectral gap assumptions. This is due to the presence of boundary layers encountered in the analysis which eventually force the bootstrap to halt before desired. We hope to address this issue in the near future. 

\smallskip

In addition to the uniform Lipschitz estimate, our approach requires higher $C^{k,1}$ estimates for $k\ge 1$, which were left essentially implicit in~\cite{AS,AM}. 
Of course, unlike the Lipschitz estimate, such higher derivative estimates for $k\geq 1$ cannot hold uniformly on the unit scale, but they are valid on \emph{mesoscopic} scales: the assertion, which is stated precisely in Theorem~\ref{t.mesoregularity} below, is roughly that a solution of~\eqref{e.pde} on the ball $B_R$, with $R\gg1$, can be well-appoximated by a $k$th degree polynomial on any mesoscopic ball $B_r(x) \subseteq B_{R/2}$ with radius $r\geq R^{1-\ep}$, for a specified  exponent $\ep >0$ which depends in particular on $k$. By ``well-approximated", we mean that the quality of the approximation of the solution by the polynomial is, up to a constant, as good as one would have for a harmonic function with the same oscillation. This result is presented below in Theorem~\ref{t.mesoregularity}. Fischer and Otto~\cite{FO} have very recently developed a higher regularity theory along somewhat different lines (see also Remark~\ref{r.FO}).

\smallskip

While we state and prove our results for scalar equations with symmetric coefficients under the assumption of finite range of dependence, we emphasize that these choices are not imposed on us by any limitation of our method. Indeed, the arguments work verbatim for systems with just changes to the notation, the techniques of~\cite{AM} can be used to get rid of the symmetry assumption, and it is precisely under weaker mixing conditions that our estimates would be optimal, since the weak mixing stops the bootstrap before it sees boundary layers. Our choice to make these additional assumptions rather reflects a desire to maintain the readability of the paper by presenting the ideas in their simplest setting.

\subsection{Assumptions}
We work in the Euclidean space $\Rd$ in dimension $d \geq 2$ and with a fixed ellipticity constant $\Lambda \geq 1$. We consider the space of coefficient fields~$\a(\cdot)$ valued in the symmetric $d$-by-$d$ matrices satisfying, for all $\xi\in\Rd$, 
\begin{equation}
\label{e.ue}
\left| \xi\right|^2 \leq \xi \cdot \a(x) \xi \leq \Lambda \left| \xi\right|^2. 
\end{equation}
We define $\Omega$ to be the set of all such coefficient fields:
\begin{equation*}
\Omega:= \left\{ \a(\cdot)\,:\, \a:\Rd \to \R^{d\times d} \ \mbox{is Lebesgue measurable, satisfies~\eqref{e.ue} and $\a^t=\a$} \right\}.
\end{equation*}
We endow $\Omega$ with the translation group $\{T_y\}_{y\in\Rd}$, which acts on $\Omega$ via
\begin{equation*}
(T_y\a)(x) := \a(x+y),
\end{equation*}
and the family $\{ \F(U)\}$ of $\sigma$--algebras on $\Omega$, with $\F(U)$ defined for each Borel subset $U\subseteq \Rd$ by
\begin{multline*} \label{}
\F(U) := \mbox{$\sigma$--algebra on $\Omega$ generated by the family of maps} \\
\a \mapsto \int_{U} q\cdot \a(x)p \, \varphi(x) \,dx, \quad p,q \in\Rd, \ \varphi\in C^\infty_c(\Rd).
\end{multline*}
Roughly,~$\F(U)$ contains the information about the coefficients restricted to~$U$. We denote the largest of these~$\sigma$--algebras by~$\F:=\F({\Rd})$. The translation group may be naturally extended to~$\F$ itself by defining
\begin{equation*}
T_yA:= \left\{ T_y\a\,:\, \a\in A \right\}, \quad A\in\F
\end{equation*}
and to any random element $X$ by setting $(T_zX)(\a):= X(T_z\a)$. 

\smallskip

Throughout the paper, we consider a probability measure~$\P$ on $(\Omega,\F)$ which is assumed to satisfy the following two conditions:
\begin{enumerate}

\item[(P1)] $\P$ is stationary with respect to $\Zd$--translations: for every $z\in \Zd$ and $A\in \F$,
\begin{equation*} 
\P \left[ A \right] = \P \left[ T_z A \right].
\end{equation*}

\smallskip

\item[(P2)] $\P$ has a unit range of dependence: for every pair of Borel subsets $U, V\subseteq \Rd$ with $\dist(U,V) \geq 1$, 
\begin{equation*}
\mbox{$\F(U)$ \  and \ $\F(V)$ \ are \ $\P$--independent.}
\end{equation*}
\end{enumerate}

The expectation of an $\F$-measurable random variable $X$ with respect to $\P$ is denoted by~$\E\left[X \right]$.

\subsection{Notation}
We continue with some notation used throughout the paper. For a measurable set $E\subseteq \Rd$, we denote the Lebesgue measure of $E$ by $|E|$ unless $E$ is a finite set, in which case $|E|$ denotes the cardinality of $E$. For a bounded Lipschitz domain $U \subset \Rd$ with $|U| < \infty$ and $p\in [1,\infty)$, we denote the normalized $L^p(U)$ norm of a function $f\in L^p(U)$ by 
\begin{equation} 
\label{e.def.nL}
\| f \|_{\underline{L}^p(U)} := \left( \fint_U \left| f(x )\right|^p\,dx\right)^{\frac1p}.
\end{equation}
It is also convenient to denote $\| f \|_{\underline{L}^\infty(U)} := \| f \|_{L^\infty(U)}$. For a vector-valued $F\in L^p(U;\Rd)$, we write $\| F \|_{\underline{L}^p(U)}:= \| | F|  \|_{\underline{L}^p(U)}$. The average of a function $f\in L^1(U)$ on $U$ is denoted by
\begin{equation*} \label{}
\left( f \right)_U:= \fint_U f(x) \,dx. 
\end{equation*}
For $U\subseteq\Rd$ with $|U|<\infty$ and vector fields $F,G\in L^2(U;\Rd)$, we denote 
\begin{equation}
\label{e.L2ip}
\innerply{F}{G}{U}:= \fint_{U} F(x) \cdot G(x)\,dx.
\end{equation}
Since we also wish to quantify the convergence of various functions in the weak~$L^p$ topology, it is natural to work with $W^{-1,p}$ norms. For this purpose we introduce the normalized $W^{-1,p}$ norm of $F \in L^p(U; \Rd)$ by
\begin{equation*}
\|F\|_{\nW^{-1,p}(U)} := \sup\left\{ \left| \innerply{F}{\eta}{U}\right| \,:\, \eta \in W^{1,p'}(U;\Rd),\, \left( \eta \right)_U = 0,\,  \left\| \nabla \eta(x) \right\|_{\underline{L}^{p'}(U)} = 1
\right\}.
\end{equation*}
As usual, $p'$ denotes the H\"older conjugate of an exponent $p\in [1,\infty]$. We also use the shorthand notation
\begin{equation}
\label{e.def.H-1}
\|F\|_{\nH^{-1}(U)} := \|F\|_{\nW^{-1,2}(U)} .
\end{equation}
Note that $W^{-1,p'}$ and $H^{-1}$ are (somewhat unconventionally) used to denote the duals of $W^{1,p}/ \R$ and $H^1/ \R$, rather than $W_0^{1,p}$ and $H^1_0$. We denote cubes of side length $R>0$ by
\begin{equation*}
\cu_R=\cu_R(0):= \left( -\frac12R,\frac12R\right)^d, \qquad \cu_R(x):= x + \cu_R.
\end{equation*}
The family of cubes of side length at least one is
\begin{equation} 
\label{e.def.mclC}
\mathcal C:= \left\{ \cu_R(x)\,:\, x\in\Rd, \, R\geq 1\right\}. 
\end{equation}
The set of (real-valued) polynomials on~$\Rd$ with degree at most $k\in\N$ is denoted by $\mathcal{P}_k$. The (random) vector space of solutions of~\eqref{e.pde} in $U\subseteq \Rd$ is 
\begin{equation} \label{e.AU}
\mathcal{A}(U):= \left\{ u \in H^1_{\mathrm{loc}}(U) \, : \, \forall v\in H^1_0(U), \ \int_U \nabla v(x) \cdot \a(x) \nabla u(x) \,dx  = 0 \right\}.
\end{equation}
Recall that, for each $p\in\Rd$, the \emph{corrector} $\Phi(\cdot,p)$ is defined for $\P$--almost every~$\a\in\Omega$ as the unique function (up to a constant) in $H^1_{\mathrm{loc}}(\Rd)$ with a $\Zd$--stationary, mean-zero gradient and which satisfies the equation 
\begin{equation*} \label{}
-\nabla \cdot \left( \a(x) \left(p+\nabla \Phi(\cdot,p)\right)\right) = 0 \quad \mbox{in}  \ \Rd.
\end{equation*}
See~\cite{JKO} for details. 

\subsection{Subadditive quantities and main results}
\label{ss.subadd}
We now recall some objects from~\cite{AS,AM} which play a central role in the paper. 
Given a bounded Lipschitz domain $U\subseteq \Rd$ and $p, q \in \Rd$, we denote the linear function of slope $p$ by $\ell_p(x):=p\cdot x$  and define
\begin{equation*}
\nu(U,p):= \inf_{v\in \ell_p + H^1_0(U)} \fint_U \frac12 \nabla v(x) \cdot \a(x) \nabla v(x)\,dx,
\end{equation*}
which is the natural subadditive quantity representing the energy of the solution of the ``cell problem," introduced (in a more general form) by Dal Maso and Modica~\cite{DM1,DM2} in their proof of qualitative homogenization of convex integral functionals. We also define
\begin{equation*}
\mu(U,q) = \inf_{u\in H^1(U)} \fint_U \left( \frac12 \nabla u(x) \cdot \a(x) \nabla u(x) - q\cdot \nabla u(x) \right)\,dx,
\end{equation*}
which is the natural superadditive quantity introduced in~\cite{AS} (also in a more general form), and is dual to $\nu(U,p)$.

\smallskip

The quantity $\nu(\,\cdot\,,p)$ is \emph{subadditive} in the sense that if $U,U_1,\ldots,U_k$ are bounded domains satisfying
\begin{equation*} 
U_1,\ldots,U_k \quad \mbox{are pairwise disjoint and} \quad \left| U\setminus (U_1 \cup\cdots \cup U_k) \right| = 0, 
\end{equation*}
then
\begin{equation} 
\label{e.subadd}
\nu(U,p) \leq \sum_{i=1}^k \frac{|U_i|}{|U|} \nu(U_i,p).
\end{equation}
This is immediate from the fact that a candidate minimizer for $\nu(U,p)$ can be obtained by assembling the minimizers of the quantities $\nu(U_i,p)$, as these agree on the boundaries. (We remark that we are using the term ``subadditive"  in a nonconventional way, due to our normalization, as it would be more standard to say that $U \mapsto |U|\, \nu(U,p)$ is subadditive.) Conversely, the quantity $\mu(\,\cdot\,,q)$ is \emph{superadditive} (that is, $-\mu(\,\cdot\,,q)$ is subadditive), since the minimizer for $\mu(U,q)$ provides with a candidate minimizer for each $\mu(U_i,q)$ by restriction. 

\smallskip

Note that the minimizer for~$\mu(U,q)$ will be the solution of a Neumann problem for~\eqref{e.pde} in $U$, while of course the minimizer for $\nu(U,p)$ is the solution of a Dirichlet problem. Thus we may think of $\nu(U,p)$ as ``the subadditive quantity for gradients" and of $\mu(U,q)$ as ``the superadditive quantity for fluxes." 

\smallskip

The homogenized matrix $\ahom$ can be defined as the symmetric matrix satisfying
\begin{equation}
\label{e.ahomdef}
\frac 1 2 p \cdot \ahom p = \lim_{R \to \infty} \E[\nu(\cu_R,p)].
\end{equation}
We view the quantities $\mu$ and $\nu$ as central to the quantitative theory of homogenization. The main step in the quantitative arguments in~\cite{AS,AM} is to show roughly that there exists $\alpha(d,\Lambda)>0$ such that, for every $\cu\in \mcl C$,
\begin{equation*} \label{}
\Ll| \nu(\cu,p)  - \tfrac 1 2 p \cdot \ahom p  \Rr| \leq C|p|^2\left| \cu\right|^{-\alpha} \quad \mbox{with overwhelming probability,}
\end{equation*}
and a similar estimate for $\mu$,
\begin{equation*} \label{}
\Ll|\mu(\cu,q) +\tfrac 1 2 q \cdot \ahom^{-1} q  \Rr| \leq C|q|^2\left| \cu\right|^{-\alpha} \quad \mbox{with overwhelming probability.}
\end{equation*}
The first main result of this paper improves the rate of convergence from the unspecified and tiny $\alpha(d,\Lambda)>0$ in~\cite{AS} to any exponent $\alpha < \frac1d$. 

\begin{theorem}[Rate of convergence of subadditive quantities]
Let $\alpha < \frac1d$. There exists $C(d,\Lambda, \alpha) < \infty$ such that for every $p,q \in \R^d$, $\cu \in \mcl C$ and $\lambda \in \R$,
\begin{equation}
\label{e.conv}
\left\{
\begin{aligned}
& \log \E\Ll[\exp \Ll(\lambda|p|^{-2}|\cu|^\alpha  \Ll|\nu(\cu,p) - \tfrac 1 2 p \cdot \ahom p  \Rr| \Rr) \Rr] \le C(1+\lambda^2), \quad \mbox{and} \\
& \log \E\Ll[\exp \Ll(\lambda|q|^{-2}|\cu|^\alpha \Ll|\mu(\cu,q) + \tfrac 1 2 q \cdot \ahom^{-1} q  \Rr|\Rr) \Rr] \le C(1+\lambda^2).
\end{aligned} 
\right.
\end{equation}
\label{t.conv}
\end{theorem}

Testing the definition of $\mu(U,q)$ with the minimizer of $\nu(U,p)$ yields
\begin{equation} 
\label{e.defJ}
J(U,p,q) := \nu(U,p) - \mu(U,q) -p\cdot q \geq 0. 
\end{equation}
Moreover, as we will see, $J(U,p,q)$ can be expressed in a variational form, and the optimizer is precisely the difference of the minimizers of $\mu(U,q)$ and $\nu(U,p)$. In other words, the difference of the subadditive energies is an energy of the differences of the minimizers. The proof of Theorem~\ref{t.conv} is based on a bootstrap argument, using the higher regularity estimates, applied to this difference, to accelerate the convergence of $J(\cu,p,\ahom p)$ to zero. See Section~\ref{ss.sketch} for a sketch of the bootstrap argument. 

\smallskip

The estimate~\eqref{e.conv} is optimal in the sense that it is false for any $\alpha > \frac1d$. To see this, we note that minimizers for $\mu$ satisfy (oscillating) Neumann conditions while those of $\nu$ satisfy Dirichlet conditions. Therefore, the gradient of the difference of the minimizers of $\mu(\cu,\ahom p)$ and $\nu(\cu,p)$ will in general be $O(1)$ in a boundary layer of at least unit thickness. The proportion of volume of such a boundary layer relative to the whole cube $\cu$ is of order $|\cu|^{-1/d}$. Therefore $J(\cu,p,\ahom p)$ should be at least $c|\cu|^{-1/d}$ in general.

\smallskip

It is the presence of this boundary layer that so far limits our quantitative results to consequences of~\eqref{e.conv}. On the other hand, one  expects the minimizers of $\nu(U,p)$ and $\mu(U,\ahom p)$ to be much closer in the interior of $U$ than near the boundary, and so if the boundary layer could be neglected, there is hope to prove much more precise results. 

\smallskip

Theorem~\ref{t.conv} can be used to prove an array of quantitative estimates in homogenization with exponents which, under the strongest mixing assumptions, will typically differ from the optimal one by a square root. For example, one can show that the $L^2$ error in homogenization is $O(\ep^{1/2-})$, where the microscopic length scale is $\ep$, or that the $H^1$ norm of the two-scale expansion is $O(\ep^{1/4-})$. The application we present here is an estimate of the sublinearity of the corrector and of the $H^{-1}$ norm of its gradient (recall that the $H^{-1}$ norm measures weak convergence in $L^2$). It roughly states that 
\begin{equation*} \label{}
 \left\| \Phi(\cdot,p) -\left( \Phi(\cdot,p) \right)_{B_R} \right\|_{\underline{L}^2(B_R)} + \|\nabla \Phi(\cdot ,p)\|_{\nH^{-1}(B_R)}
\lesssim R^{\frac12+},
\end{equation*}
with very strong stochastic integrability. Note that estimates on the sublinearity of the corrector are intimately connected to estimates for the error in homogenization.

\begin{theorem}[Sublinear growth of the corrector]
\label{t.correctors}
Let $\beta \in\left(0,\frac12\right)$. There exists a constant $C(d,\Lambda,\beta) < \infty$ such that, for every $p\in\Rd$, $R \ge 1$ and $\lambda \in \R$, 
\begin{multline*} \label{}
\log \E\Ll[\exp \Ll(\lambda |p|^{-2} R^{-2+2\beta} \left( \left\| \Phi(\cdot,p) -\left( \Phi(\cdot,p) \right)_{B_R} \right\|_{\underline{L}^2(B_R)}^2 +  \|\nabla \Phi(\cdot ,p)\|_{\nH^{-1}(B_R)}^2\right)   \Rr) \Rr] \\
 \leq C(1+\lambda^2).
\end{multline*}
\end{theorem}

Theorem~\ref{t.correctors} is new, even though it is almost certainly suboptimal and does not compare favorably with the results of Gloria and Otto~\cite{GO3} under spectral gap assumptions, who proved that correctors are almost bounded (see also Gloria, Neukamm and Otto~\cite{GNO3} for more results under spectral gap assumptions). Indeed, it seems to be an open question (and there is some doubt) whether finite range of dependence implies a spectral gap-type inequality in $d>1$. The best previous result for finite range of dependence was for $\beta(d,\Lambda)>0$ very small (cf.~\cite{AS}). Moreover, as mentioned above, our methods are applicable under essentially any mixing condition and we believe yield essentially optimal estimates under weaker mixing conditions (such as mixing conditions so slow that spatial averages of the coefficients on scale $R$ converge slower than $R^{-1}$ to their mean). This will be explained in more details in future papers. 

\smallskip

The novelty of Theorem~\ref{t.correctors} is not however in its statement, but in its proof.

\subsection{Outline of the paper}

Our results rely crucially on $C^{k,1}$ estimates for solutions of \eqref{e.pde} on \emph{mesoscopic} scales. This was essentially proved in~\cite{AS}, where the higher regularity theory in stochastic homogenization was introduced; the precise statements we require are proved in Section~\ref{s.higherreg}. In Section~\ref{s.propertiesJ}, we make the fundamental observation that the quantity $J$ in \eqref{e.defJ} can be expressed as a ``modulated energy" of the difference of the minimizers of $\mu$ and $\nu$. 

\smallskip

We break the proof of Theorem~\ref{t.conv} into two main steps. First, in Section~\ref{s.boostrap.cubes}, we use an induction argument on the exponent $\alpha$ to show that, for every $\alpha< \frac1d$, $p \in \R^d$ and cube $\cu \in \mcl C$,
\begin{equation}
\label{e.conv-E}
\Ll| \E[\nu(\cu,p)] - \frac 1 2 p \cdot \ahom p \Rr| \le C |p|^2 \, |\cu|^{-\alpha},
\end{equation}	
and a similar estimate with $\mu$ in place of $\nu$. This statement is \emph{a priori} much weaker than the conclusion of Theorem~\ref{t.conv}, but we show in Section~\ref{s.magic} that the subadditivity of $-\mu$ and $\nu$ can be used to upgrade the stochastic integrability from $L^1$ to exponential moments without sacrificing any of the exponent.

\smallskip

Finally, in Section~\ref{s.sublin}, we present a functional inequality, which appears to be new and which we call \emph{multiscale Poincar\'e inequality}. It has a wavelet flavor and converts control over the spatial averages of the gradient of a function in triadic subcubes into an estimate on the oscillation of the function itself. This, together with the regularity theory, reduces Theorem~\ref{t.correctors} to Theorem~\ref{t.conv}.

\section{Higher regularity on mesoscopic scales}
\label{s.higherreg}

A cornerstone of the methods in this paper is a quenched $C^{k,1}$ estimate for solutions of~\eqref{e.pde} on \emph{mesoscopic} scales. This was essentially proved in~\cite{AS}, where the higher regularity theory in stochastic homogenization was introduced.

\smallskip

The rough statement, which is given precisely in the following theorem, asserts that a solution of~\eqref{e.pde} can be approximated, just as well as a harmonic function, by $k$th degree harmonic polynomials -- on length scales larger than a fixed \emph{mesoscopic} scale (or the microscopic scale, if  $k=0$). We let $\Ahom_k$ denote the set of polynomials of degree at most~$k$ which are $\ahom$-harmonic.

\begin{theorem}[$C^{k,1}$ regularity on mesoscopic scales]
\label{t.mesoregularity}
Fix $s\in (0,d)$ and $k\in\N$. There exist an exponent $\delta(s,d,\Lambda)>0$, a constant $C(s,k,d,\Lambda) < \infty$ and an $\F$-measurable random variable $\mathcal{X}:\Omega\to [1,\infty)$ satisfying the estimate
\begin{equation}
\label{e.minimalradius}
\E \left[ \exp\left( \mathcal{X}^s \right) \right] < \infty
\end{equation}
and such that, for every $R\geq 2\X$, $v \in \A(B_R)$ and $r \in \left[ \mathcal{X} \vee R^{\frac{k}{k+\delta}} ,\frac12R\right]$,
\begin{equation}
\label{e.mesoregularity} 
\inf_{w\in \Ahom_k}\left\| v-w \right\|_{\underline{L}^2(B_r)} \leq C \left( \frac rR \right)^{k+1} \left\| v \right\|_{\underline{L}^2(B_R)}.
\end{equation}
\end{theorem}

Theorem~\ref{t.mesoregularity} is essentially proved in~\cite{AS}. That paper only stated the (uniform) $C^{0,1}$ estimate and left (mesoscopic) higher regularity statements for~$k\geq 1$ implicit in the proof. To prove the proposition, we just need to expound the argument from~\cite{AS}. 

\smallskip

A second ingredient in the proof of Theorem~\ref{t.mesoregularity}, which also plays a key role in the bootstrap argument, is the following quenched error estimate for the Dirichlet problems proved in~\cite{AS}. It gives a quenched, deterministic estimate for the error in homogenization for the Dirichlet problem on length scales larger than a certain random scale of characteristic size $O(1)$. 

\begin{proposition}[{\cite[Theorem 1.1]{AS}}]
\label{p.quenchedEE}
Fix $s\in (0,d)$, $\ep>0$, and let $U \subseteq B_1$ be a Lipschitz domain. There exists an exponent $\delta(d,\Lambda,s,\ep)>0$, a constant $C(s,d,\Lambda,U) < \infty$ and an $\F$-measurable random variable $\mathcal{R}:\Omega\to [1,\infty)$, which depends only on $(d,\Lambda,s,\ep)$ and satisfies the estimate
\begin{equation}
\label{e.minimalradiusEE}
\E \left[ \exp\left( \mathcal{R}^s \right) \right]  < \infty,
\end{equation}such that, for every $r\geq \mathcal{R}$, $f\in W^{1,2+\ep}(rU)$ and solutions $u,\overline{u} \in f+H^1_0(rU)$ of 
\begin{equation*}
-\nabla \cdot \left( \a\nabla u \right) = 0 \quad \mbox{and} \quad -\nabla \cdot \left( \ahom \nabla \overline{u} \right) = 0 \quad \mbox{in} \ rU,
\end{equation*}
we have the estimate
\begin{equation}
\label{e.quenchedEE}
\frac1{r}  \left\| u - \overline{u} \right\|_{\underline{L}^2(rU)} \leq r^{-\delta} \left\| \nabla f \right\|_{\underline{L}^{2+\ep}(rU)}. 
\end{equation}
\end{proposition}

\begin{remark}
\label{r.FO}
The mesoscopic regularity estimate in Theorem~\ref{t.mesoregularity} implies Liouville theorems of all orders. These state roughly that, for each $k\in\N$, the subspace of $\A(\Rd)$ consisting of functions which grow at most like $o\left( |x|^{k+1} \right)$ has the same dimension as $\Ahom_k$. The latter functions correspond precisely to the \emph{$k$th order correctors}. Note that, in view of the results of~\cite{AM}, Theorem~\ref{t.mesoregularity} and therefore these Liouville theorems hold in much greater generality than we present here (e.g., for nonlinear equations and with weaker mixing conditions). A similar result was recently proved by Fischer and Otto~\cite{FO}, who developed a similar higher regularity theory in the general stationary ergodic setting. Similar to~\cite{GNO}, they proceed a bit differently: rather than measure the distance of an element of $\A(B_R)$ to $\Ahom_k$, they measure its distance to the $k$th order correctors. We believe that while both approaches are of interest, the one presented here is  more faithful to what a $C^{k,1}$ estimate should be in this context (we would like, for example, to measure the $C^{k,1}$ seminorm of a $j$th order corrector, for $j\leq k$, and not get zero, otherwise we gain no information on the correctors themselves). See below Remark~4.5 in~\cite{ASh} for some similar comments.
\end{remark}

That a quantitative estimate for the error in homogenization for the Dirichlet problem implies higher regularity estimates via a Campanato iteration was an idea introduced in~\cite{AS}. Here we formulate a version in the following lemma, which is a variation of \cite[Lemma 5.1]{AS}.

\begin{lemma}
\label{l.regularity}
Fix $R\geq 2$, $\alpha>0$, $p\in [1,\infty]$ and $u \in L^p(B_R)$. For each $k\in\N$ and $s\in \left(0, R\right]$, denote
\begin{equation*} \label{}
D_k(s):= \inf_{w \in \Ahom_k} \left\| u-w \right\|_{\underline{L}^p(B_s)}  
\end{equation*}
Assume that $h \in \left[1,\frac12R\right]$ and  have the property that, for every $r \in \left[h,\frac12R\right]$, there exists an  $\ahom$-harmonic function $v\in C^\infty(B_r)$ such that 
\begin{equation} \label{e.harmonicapprox}
\left\| u-v \right\|_{\underline{L}^p(B_r)} 
\leq r^{-\alpha} D_0(2r). 
\end{equation}
Then, for each $k\in\N$, there exists a constant $C(d,\Lambda,k,\alpha) <\infty$ such that, for every $r \in \left[ h, \frac12R \right]$,
\begin{equation} \label{e.Ck1regularity}
D_k(r) \leq C \left(\frac r R  \right)^{k+1} D_k(R) + C r^{-\alpha} \left( \frac rR \right) D_0(R).
\end{equation}

%
\end{lemma}
\begin{proof}
Fix $k\in\N$. Throughout, we denote by $C$ and $c$ positive constants depending only on $(k,\alpha,\Lambda)$ which may vary in each occurrence. In the first few steps, we make some preliminary observations and introduce the  notation we need in the main part of the argument, which begins in Step~3.

\smallskip

\noindent \emph{Step 1.}  We first observe that the triangle inequality and the hypothesis~\eqref{e.harmonicapprox} imply that, for every $r \in \left[h, \frac12R\right]$ and $s\in \left(0, \frac12r \right]$, 
\begin{equation} \label{e.excessdecay}
D_k(s) \leq C \left( \frac sr \right)^{k+1} D_k(r)  +  C \left( \frac rs \right)^{\frac dp} r^{-\alpha} D_0(r).
\end{equation}
Selecting a harmonic function $v\in C^\infty(B_r)$ to satisfy~\eqref{e.harmonicapprox}, we find that 
\begin{align*} \label{}
D_k(s) & =  \inf_{w \in \Ahom_k} \left\| u-w \right\|_{\underline{L}^p(B_s)} \\
& \leq \inf_{w\in \Ahom_k} \left\| v-w \right\|_{\underline{L}^p(B_s)} +  \left\| u-v \right\|_{\underline{L}^p(B_s)}  \\
& \leq  C\left( \frac sr \right)^{k+1}  \inf_{w \in \Ahom_k}\left\| v-w \right\|_{\underline{L}^p(B_r)} + Cr^{-\alpha} D_0(r) \left( \frac{r}{s}   \right)^{\frac dp} .
\end{align*} 
In the last line, we used the fact that any $\ahom$-harmonic function $v$ satisfies, for every $0<s\leq \frac12r$, 
\begin{equation*} \label{}
\inf_{w\in \Ahom_k} \left\| v-w \right\|_{L^\infty(B_s)}
\leq C\left( \frac sr \right)^{k+1} \inf_{w\in \Ahom_k} \left\| v-w \right\|_{\underline{L}^p(B_r)}.
\end{equation*}
Next we use the triangle inequality and~\eqref{e.harmonicapprox} a second time to get
\begin{equation*} \label{}
 \inf_{w\in \Ahom_k}  \left\| v-w \right\|_{\underline{L}^p(B_r)}
 \leq D_k(r) + Cr^{-\alpha} D_0(r). 
\end{equation*}
Substituting into the inequality above, we get 
\begin{align*} \label{}
D_k(s) 
& \leq C \left( \frac sr \right)^{k+1} \left( D_k(r) +  r^{-\alpha} D_0(r) \right) + C\left( \frac rs \right)^{\frac dp} r^{-\alpha} D_0(r)  \\
& \leq C \left( \frac sr \right)^{k+1} D_k(r)  + C \left( \frac rs \right)^{\frac dp} r^{-\alpha} D_0(r). 
\end{align*}

\smallskip

\noindent \emph{Step 2.} We set up the rest of the argument. Fix $\theta = \theta(k)\geq c$ so small that $C \theta = \frac14$, where~$C$ is the constant in~\eqref{e.excessdecay}, so that the latter inequality implies
\begin{equation*} \label{}
\tilde D_k(\theta r ) \leq \frac12 \tilde D_k(r) + C r^{-k-\alpha} D_0(r),
\end{equation*}
where here and in what follows we set $\tilde D_k(s):= s^{-k} D_k(s)$.  An iteration of the previous inequality gives
\begin{equation} \label{e.edpause1}
\tilde D_k(\theta^m r) \leq 2^{-m} \tilde D_k(r) + C \sum_{j=0}^{m-1} 2^{j-m} \left( \theta^j r \right)^{-k-\alpha} D_0\left( \theta^j r\right)
\end{equation}
provided $\theta^{m-1} r \geq h$. To shorten the notation, we denote $r_j := \theta^j R$, $B^j := B_{r_j}$ and by $w_{k,j}$ the best $k${th} degree polynomial approximation of $u$ in $B^j$, that is, $w_{k,j}$ satisfies
\begin{equation*} 
\left\| u - w_{k,j} \right\|_{\underline{L}^p(B^j)} = \inf_{w \in \mathcal{P}_k}  \left\| u - w \right\|_{\underline{L}^p(B^j)}
\end{equation*}

\smallskip

\noindent \emph{Step 3.} We now complete the proof of the proposition in the case $k=0$. In fact, this has been already proved in~\cite[Lemma 5.1]{AS}, but we give the argument for the sake of completeness. 

\smallskip

First, since we may add constants to both $u$ and $v$, we may assume without loss of generality that $w_{0,0} = 0$. It follows that 
\begin{equation*} 
\left\|  w_{1,0}   \right\|_{\underline{L}^p(B_R)}
 \leq 2D_0(R),
\end{equation*}
and hence we easily get 
\begin{equation} \label{e.edpause31}
\left\|w_{1,0}\right\|_{L^\infty(B_R)} + R\left\|\nabla w_{1,0}\right\|_{L^\infty(B_R)} \leq C D_0(R).
\end{equation}
Using~\eqref{e.edpause1} we obtain
\begin{align*} 
 \frac1{r_{m+1}}  \left\| w_{1,m+1} - w_{1,m} \right\|_{\underline{L}^p(B^{m+1}) } 
&  \leq \tilde D_1(r_{m+1}) +   \theta^{-1-\frac dp} \tilde D_1(r_{m}) \\
& \leq 2^{-m} C \tilde D_1(R) + C \sum_{j=0}^{m}   2^{j-m} r_j^{-1-\alpha} D_0\left( r_j \right).
\end{align*}
We thus deduce that 
\begin{equation*} \label{}
\left|\nabla w_{1,m+1} - \nabla w_{1,m}\right| \leq 2^{-m} C \tilde D_1(R) + C \sum_{j=0}^{m}   2^{j-m}  r_j^{-\alpha} \frac{D_0\left( r_j \right)}{r_j},
\end{equation*}
and it follows by summation that 
\begin{equation*} \label{}
\left|\nabla w_{1,n} - \nabla w_{1,0}\right| \leq C \tilde D_1(R)  + C \sum_{j=0}^{n}  r_j^{-\alpha} \frac{D_0\left( r_j \right)}{r_j}
\end{equation*}
provided that $r_n \geq h$. Combining this with~\eqref{e.edpause1} and~\eqref{e.edpause31} leads to  
\begin{equation*} 
\frac{D_0(r_n)}{r_n}  \leq  \tilde D_1(r_n) +   \left|\nabla w_{1,n} \right| \leq  C \frac{D_0(R)}{R} + C \sum_{j=0}^{n-1}  r_j^{-\alpha} \frac{D_0\left(r_j \right)}{r_j} .
\end{equation*}
Taking supremum then gives 
\[
\sup_{m \in \{0,\ldots ,n\}} \frac{D_0(r_m)}{r_m}  \leq C \frac{D_0(R)}{R} +  C \sum_{j=0}^{n}  r_j^{-\alpha}  \sup_{m \in \{0,\ldots ,n\}} \frac{D_0(r_m)}{r_m} 
\]
provided that $r_n \geq h$. Letting now $n^* = n^*(d,\alpha)$ be the largest integer such that $ C \sum_{j=0}^{n^*}  r_j^{-\alpha} \leq \frac12$ and $r_{n^*}\geq h$ 
we obtain after straightforward manipulations that
\begin{equation} \label{e:mesoregconcl1}
\sup_{h \leq r \leq R} \frac{D_0(r)}{r} \leq C \frac{D_0(R)}{R}.
\end{equation}

\smallskip

\noindent \emph{Step 4.} We complete the argument in the case of general $k\in\N$. 
Following the reasoning of the previous step, letting $\tilde w_{k+1,j}$ stand for the $k^{th}$ order polynomial part of $w_{k+1,j}$, we get, by~\eqref{e.edpause1} and~\eqref{e:mesoregconcl1},
\begin{align}  \label{e.edpause42} 
 \tilde D_{k}(r_{j}) 
& \leq \frac 1{r_{j}^{k}} \left\| u - \tilde w_{k+1,j} \right\|_{\underline{L}^p(B^j)} \\ \nonumber 
& \leq r_j \tilde D_{k+1}(r_{j}) + \frac 1{r_{j}^{k}}\left\| w_{k+1,j} - \tilde w_{k+1,j} \right\|_{\underline{L}^p(B^j)} \\
& \leq r_j \left( 2^{-j} \tilde D_{k+1}(R) + C r_j^{-k-\alpha} \frac{D_0(R)}{R} + \left|\nabla^{k+1} w_{k+1,j}\right|  \right).\nonumber
\end{align}
Thus we are left to estimate $\left|\nabla^{k+1} w_{k+1,j} \right| $. Since 
\begin{align*} 
\frac 1{r_{j+1}^{k+1}} \left\|w_{k+1,j+1} -  w_{k+1,j+1} \right\|_{\underline{L}^p(B^{j+1})}
& \leq C\left(\tilde D_{k+1}(r_{j+1}) + \tilde D_{k+1}(r_j)  \right) \\ 
& \leq C 2^{-j} \tilde D_{k+1}(R) + C r_j^{-k-\alpha} \frac{D_0(R)}{R},
\end{align*}
we deduce that 
\begin{equation} \label{e.edpause43} 
\left|\nabla^{k+1} w_{k+1,j+1} - \nabla^{k+1} w_{k+1,j}\right| \leq C 2^{-j} \tilde D_{k+1}(R) + C r_j^{-k-\alpha} \frac{D_0(R)}{R}.
\end{equation}
Indeed, for any $n\in\N$ and polynomial $\phi\in\mathcal{P}_n$, we have that 
\begin{equation} \label{e.polynomials silly}
\left( \fint_{B_r} \left|\phi(x)\right|^p \, dx  \right)^{\frac 1p} = \left( \fint_{B_1} \left|\phi(r x)\right|^p \, dx  \right)^{\frac 1p} \geq c \sup_{x \in B_1} \left|\phi(r x)\right| \geq c  r^{n} \left|\nabla^n \phi\right|
\end{equation}
for a constant $c = c(d,n)>0$. This is due to scaling and Lemma~\ref{l.Markov}, which is stated and proved below. 

\smallskip

Now~\eqref{e.edpause43} implies after summation that
\begin{equation*}
|\nabla^{k+1} w_{k+1,j} - \nabla^{k+1} w_{k+1,0}| \leq C  \tilde D_{k+1}(R) + C r_j^{-k-\alpha} \frac{D_0(R)}{R}.
\end{equation*}
Furthemore, since
\begin{align*} 
\left\|w_{k,0} -  w_{k+1,0} \right\|_{\underline{L}^p(B_R)} 
 \leq D_{k+1}(R) + D_k(R) \leq 2 D_k(R),
\end{align*}
we obtain
\begin{equation*} \label{}
\left| \nabla^{k+1} w_{k+1,0}\right| \leq C R^{-k-1} D_k(R),
\end{equation*}
and consequently
\begin{align*} 
\left|\nabla^{k+1} w_{k+1,j} \right| 
& \leq  \left|\nabla^{k+1} w_{k+1,j} - \nabla^{k+1} w_{k+1,0}\right|  + \left|\nabla^{k+1} w_{k+1,0}\right| \\
& \leq  C  R^{-k-1} D_{k}(R) + C r_j^{-k-\alpha} \frac{D_0(R)}{R}.
\end{align*}
Combining this finally with~\eqref{e.edpause42} allows us to conclude with
\begin{equation*} 
D_k(r_j) \leq C \left(\frac{r_j}{R}\right)^{k+1} D_{k}(R) + C r_j^{-\alpha} \left(\frac{r_j}{R}\right) \frac{D_0(R)}{R},
\end{equation*}
from which the statement~\eqref{e.Ck1regularity} can be easily deduced. 
\end{proof}

\begin{remark}
Notice that since $D_k(R) \leq D_0(R)$, the estimate~\eqref{e.Ck1regularity} implies
\begin{equation*} \label{}
D_k(r) \leq C \left( \frac rR \right)^{k+1} D_0(R) + Cr^{-\delta} \left( \frac rR \right) D_0(R).
\end{equation*}
This is a $C^{k,1}$ estimate on scales $r$ for which the second term on the right side is smaller than the first term, that is, for $r$ satisfying
\begin{equation*} \label{}
r\gtrsim R^{\frac{k}{k+\delta}}.
\end{equation*}
\end{remark}

In the proof of Lemma~\ref{l.regularity}, we used the following fact, which is a simple consequence of the equivalence of norms in finite-dimensional vector spaces.

\begin{lemma} 
\label{l.Markov}
Fix $k \in \N$. There exists a constant $C(d,k) < \infty$ such that, for every $m \le k$ and $w\in \mathcal{P}_k$,
\begin{equation*} \label{}
\left\| \nabla^m w\right\|_{L^\infty(B_1)} \le C \Ll\| w\Rr\|_{L^1(B_1)}.
\end{equation*}
\end{lemma}

\begin{proof}[{Proof of Theorem~\ref{t.mesoregularity}}]
The proof is essentially the same as -- indeed, even considerably simpler than (due to homogeneity) -- the argument given in the general nonlinear case considered in~\cite{AS,AM}. If we take $U=B_1$ and $\ep(d,\Lambda)>0$ to be the exponent in the interior Meyers estimate (cf.~\cite[Proposition B.6]{AM} for instance), then the solvability of the Dirichlet problem and the conclusion of Proposition~\ref{p.quenchedEE}, with $\mathcal{R}$ given there, ensures that the hypothesis of Lemma~\ref{l.regularity} is satisfied for $h = \mathcal{R}$. The lemma then yields the result for $\X=\mathcal{R}$.
\end{proof}

\section{Properties of the modulated energy~\texorpdfstring{$J$}{J}}
\label{s.propertiesJ}

We begin this section with the simple but key observation that the difference between $\nu$ and $\mu$ can be expressed as a ``modulated energy'' of the difference of the minimizers. We define $\J(\cdot,U,p,q)$ for functions $w\in H^1(U)$ by
\begin{equation*} \label{}
\J(w,U,p,q) := \fint_U  \left( - \frac12 \nabla w(x) \cdot \a(x) \nabla w(x)- p\cdot \a(x) \nabla w(x) + q\cdot \nabla w(x)  \right)\,dx.
\end{equation*}
We denote the maximum of $\J(\cdot,U,p,q)$ among solutions of the PDE by
\begin{equation*} \label{}
J(U,p,q) := \max_{w\in \A(U)} \J(w,U,p,q).
\end{equation*}
To motivate the definition of $J(U,p,q)$, we show that it is actually a familiar object: it can be decomposed into the quantities studied in~\cite{AS}, which are the focus of the analysis in that paper. 

\begin{lemma}
\label{l.modulation}
For every bounded Lipschitz domain $U\subseteq\Rd$ and $p,q\in\Rd$, 
\begin{equation} \label{e.modidentity}
J(U,p,q) = \nu(U,p) - \mu(U,q) - p\cdot q.
\end{equation}
\end{lemma}
\begin{proof}
Fix $p,q\in\Rd$, $u\in \A(U)$ and let $v \in \ell_p+H^1_0(U)$ be the minimizer in the definition of~$\nu(U,p)$. Note that $v\in\A(U)$ and compute
\begin{align*}
\lefteqn{
\nu(U,p) - \fint_U \frac12 \left( \nabla u(x) \cdot \a(x) \nabla u(x) - q\cdot \nabla u(x) \right)\,dx - p\cdot q 
} \qquad & \\
& = \fint_U \left( \frac12 \nabla v(x) \cdot \a(x) \nabla v(x) - \frac12 \nabla u(x) \cdot \a(x) \nabla u(x) + q\cdot \nabla u(x)  \right)\,dx - p\cdot q \\
& = \fint_U  \bigg( - \frac12 \left( \nabla u(x) - \nabla v(x) \right) \cdot \a(x) \left( \nabla u(x) - \nabla v(x) \right) \\
& \qquad \qquad - p\cdot \a(x) \left( \nabla u(x) - \nabla v(x)  \right) + q\cdot \left( \nabla u(x)-\nabla v(x) \right)  \bigg)\,dx \\
& = \J(u-v,U,p,q).
\end{align*}
Here we used integration by parts twice, both taking advantage of the affine boundary condition for $v$ to get
\begin{equation*} \label{}
p = \fint_U \nabla v(x)\,dx
\end{equation*}
and another, which also uses that $u-v \in \A(U)$ to get
\begin{equation*} \label{}
\fint_U p\cdot  \a(x) \left( \nabla u(x) - \nabla v(x) \right)\,dx = \fint_U \nabla v(x) \cdot  \a(x) \left( \nabla u(x) - \nabla v(x) \right)\,dx.
\end{equation*}
We deduce that 
\begin{align*} \label{}
J(U,p,q) 
& = \max_{w\in \A(U)} \J(w,U,p,q)  \\
& = \max_{u\in \A(U)} \J(u-v,U,p,q) \\
& = \max_{u\in \A(U)} \left( \nu(U,p) - \fint_U \frac12 \left( \nabla u(x) \cdot \a(x) \nabla u(x) - q\cdot \nabla u(x) \right)\,dx - p\cdot q  \right) \\
& = \nu(U,p) - \mu(U,q) - p\cdot q. \qedhere
\end{align*}
\end{proof}

Notice that the proof of Lemma~\ref{l.modulation} gave more than its statement:  namely, the maximizer in the definition of $J(U,p,q)$ is precisely the difference of the minimizers of $\mu(U,q)$ and~$\nu(U,p)$.

\smallskip

In the rest of this section, we present some basic properties of $J$ which are needed in the bootstrap argument in the next section, and we fix a bounded Lipschitz domain~$U\subseteq\Rd$ throughout. We denote the unique (up to additive constants) maximizer in the definition of~$J(U,p,q)$ by
\begin{equation*}
u(\cdot,U,p,q):= \mbox{maximizer of $\J(\cdot,U,p,q)$ among functions in $\A(U)$.}
\end{equation*}
Notice that $u(\cdot,U,0,q)$ is the minimizer for $\mu(U,q)$ and $- u(\cdot,U,p,0)$ is the minimizer for $\nu(U,p)$. 

\smallskip

We next record the first and second variations of the optimization problem implicit in the definition of $J(U,p,q)$. Henceforth, we make use of the notation~\eqref{e.L2ip}.

\begin{lemma}
\label{l.firstvar}
For every $p,q\in\Rd$ and $v \in \A(U)$, 
\begin{equation}
\label{e.firstvar}
\innerply{\a\nabla u(\cdot,U,p,q)}{\nabla v}{U} = \innerply{- \a p+ q}{\nabla v}{U} 
\end{equation}
and
\begin{equation}
\label{e.secondvar}
 J(U,p,q)  - \J(u(\cdot,U,p,q)+v,U,p,q) = \frac12 \innerply{\nabla v}{\a\nabla v}{U}.
\end{equation}
\end{lemma}
\begin{proof}
For convenience, for $t\geq 0$, set $u_t:= u(\cdot,U,p,q)+ tv$. Compute
\begin{align*}
0 
&\leq \J(u_0,U,p,q) - \J(u_t,U,p,q) \\
& = t^2 \fint_{U} \left( \frac12 \nabla v(x)\cdot\a(x)\nabla v(x) \right)\, dx \\
& \qquad  + t  \fint_{U} \left(\a(x) \nabla u_0(x) + \a(x)p - q \right) \cdot  \nabla v(x)  \,dx.
\end{align*}
Dividing by $t$ and sending $t\to 0$ gives 
\begin{equation*}
  \fint_{U} \left(\a(x) \nabla u_0(x) + \a(x) p -  q \right) \cdot  \nabla v(x)  \,dx \geq 0.
\end{equation*}
Repeating the argument with $-v$ in place of $v$ yields~\eqref{e.firstvar}. Returning to the previous identity and taking $t=1$ gives~\eqref{e.secondvar}. 
\end{proof}

For reference, we observe that by taking $v=u(\cdot,U,p,q)$ in~\eqref{e.firstvar} we obtain the identities
\begin{align} \label{e.Jidentity}
J(U,p,q) = & \frac12 \innerply{\a \nabla u(\cdot,U,p,q)}{\nabla u(\cdot,U,p,q)} {U}\\ 
\nonumber = & \frac12 \innerply{-\a p+q}{\nabla u(\cdot,U,p,q)}{U}.
\end{align}
It is easy to see from~\eqref{e.firstvar} that $(p,q) \mapsto u(\cdot,U,p,q)$ is a linear map from~$\Rd\times\Rd$ into~$\A(U)$; that is, for every $p_1,p_2,q_1,q_2\in\Rd$ and $s,t\in\R$,
\begin{equation} \label{e.minimizerslinear}
u(\cdot,U,tp_1+sp_2,tq_1+sq_2) = t u(\cdot,U,p_1,q_1)+ su(\cdot,U,p_2,q_2)
\end{equation}
Likewise, $J$ is quadratic: for every $p,q\in\Rd$ and $t>0$,
\begin{equation} \label{e.Jquadratic}
J(U,tp,tq) = t^2 J(U,p,q). 
\end{equation}
We next show that $\J(\cdot,U,p,q)$ responds quadratically to perturbations near its maximum.  

\begin{lemma}
\label{l.Jconvexity}
For every $p,q\in\Rd$ and $v,w \in \A(U)$,
\begin{equation}
\label{e.lowerUC}
 \frac14 \left\| \nabla v - \nabla w \right\|_{\underline{L}^2(U)}^2
\leq  2J(U,p,q) - \J( w,U,p,q) - \J(v,U,p,q) 
\end{equation}
and
\begin{equation}
\label{e.upperUC}
2\J( v,U,p,q) - \J(w,U,p,q) - J(U,p,q)
 \leq \frac{\Lambda}{4}  \left\| \nabla v - \nabla w \right\|_{\underline{L}^2(U)}^2.
\end{equation}
\end{lemma}
\begin{proof}
For any $v_1,v_2 \in \A(U)$,
\begin{multline*}
2\J\left(\frac{v_1+v_2}2,U,p,q\right) - \J\left(v_1,U,p,q\right) - \J\left(v_2,U,p,q\right) 
\\ = \frac14 \innerply{\a (\nabla v_1 - \nabla v_2) }{\nabla v_1 - \nabla v_2}{U}.
\end{multline*}
Then~\eqref{e.lowerUC} follows by choosing $v_1 = v$ and $v_2 = w$ and using the maximality of $J(U,p,q)$, and~\eqref{e.upperUC} follows similarly by choosing $v_1 = w$ and $v_2 = 2v -w$. 
\end{proof}

\begin{lemma}
\label{l.Kidentities}
For every $p,q,p',q'\in \Rd$, 
\begin{equation} \label{e.Kident0}
J( U, p+p', q+q' ) - J(U,p,q) - J(U,p',q') 
= \innerply{-\a p' + q' }{\nabla u(\cdot,U,p,q)}{U}.
\end{equation}
\end{lemma}
\begin{proof}
By the first variation~\eqref{e.firstvar},
 \begin{align*}
 \innerply{ \a \nabla u(\cdot,U,p,q)}{\nabla u(\cdot,U,p+p',q+q')}{U}
& =   \innerply{ - \a(p+p')   + q+q'} {\nabla u(\cdot,U,p,q)}{U} \\
& = 2 J(U,p,q) + \innerply{ - \a p'   + q'} {\nabla u(\cdot,U,p,q)}{U}
\end{align*}
and, by~\eqref{e.minimizerslinear},
 \begin{align*}
\lefteqn{
\innerply{ \a \nabla u(\cdot,U,p,q)}{\nabla u(\cdot,U,p+p',q+q')}{U}
} \qquad & \\
& =   \innerply{ - \a p   + q} {\nabla u(\cdot,U,p+p',q+q')}{U} \\
& = 2 J(U,p+p',q+q') - \innerply{ - \a p'   + q'} {\nabla u(\cdot,U,p+p',q+q')}{U} \\
& = 2 J(U,p+p',q+q') - 2J(U,p',q') - \innerply{ - \a p'   + q'} {\nabla u(\cdot,U,p,q)}{U}.
\end{align*}
Combining the two displays above gives~\eqref{e.Kident0}. 
\end{proof}

We denote by $\nabla J(U,p,q)$ the gradient of $J(U,\cdot)$ at $(p,q)$. By~\eqref{e.Jquadratic}, this is a linear mapping from $\Rd\times\Rd\to\R$ which, in view of the previous lemma, can be expressed by
\begin{equation}
\label{e.Kidentity}
\nabla J(U,p,q)(p',q') = \innerply{-\a p' + q'}{\nabla u(\cdot,U,p,q)}{U}. 
\end{equation}
This relation between~$\nabla J(U,p,q)$ and the spatial averages of the gradient and flux of $u(\cdot,U,p,q)$ will play an important role in the proof of Proposition~\ref{p.bootstrap.cubes}. For simplicity, the maps $p' \mapsto \nabla J(U,p,q)(p',0)$ and $q'\mapsto \nabla J(U,p,q)(0,q')$ are sometimes denoted by $\nabla_pJ(U,p,q)$ and $\nabla_q J(U,p,q)$, respectively. 

\smallskip

Lemma~\ref{l.Kidentities} gives the identity
\begin{multline} \label{e.Jconv basic}
\frac12 J(U,p_1,q_1) + \frac12 J(U,p_2,q_2) -  J \left( U,\frac{p_1+p_2}{2}, \frac{q_1+q_2}{2} \right) \\ = \frac14 J(U,p_1-p_2,q_1-q_2)\,.
\end{multline}
This readily implies the following upper convexity estimate for $J$ in $(p,q)$ and the lower convexity estimates for $J$ in the variables $p$ and $q$ separately. 

\begin{lemma} \label{l.Jconvex}
For every $p,p_1,p_2,q,q_1,q_2\in \R^d$,
\begin{multline}
\label{e.Jupconvex}
\frac12 J(U,p_1,q_1) + \frac12 J(U,p_2,q_2) - J\left(U,\frac12(p_1+p_2) ,\frac12(q_1+q_2) \right) \\
\leq \Lambda \left( \left| p_1-p_2 \right|^2 +  \left| q_1  -  q_2  \right|^2 \right),
\end{multline}
\begin{equation} \label{e.Jconvexp}
\frac12 J(U,p_1,q) + \frac12 J(U,p_2,q) - J\left(U,\frac12(p_1+p_2) ,q \right) \geq \frac12 \left| p_1-p_2 \right|^2
\end{equation}
and 
\begin{equation} \label{e.Jconvexq}
\frac12 J(U,p,q_1) + \frac12 J(U,p,q_2) - J\left(U,p,\frac12(q_1+q_2)  \right) \geq \frac1{2\Lambda} \left| q_1-q_2 \right|^2.
\end{equation}
\end{lemma}
\begin{proof}
We apply~\eqref{e.Jconv basic} together with~\eqref{e.Jidentity} to obtain an upper bound for $J(U,p_1-p_2,q_1-q_2)$, and rearrange to obtain~\eqref{e.Jupconvex}. We obtain~\eqref{e.Jconvexp} and~\eqref{e.Jconvexq} in a similar way, using the bounds 
\begin{equation*} \label{}
J(U,p,0) = \nu(U,p) \geq \frac12|p|^2 \quad \mbox{and} \quad J(U,0,q) = -\mu(U,q) \geq \frac1{2\Lambda}|q|^2.\qedhere
\end{equation*}
\end{proof}

\smallskip

By subadditivity, stationarity and \eqref{e.ahomdef}, for every $\cu \in \mcl C$, we have
\begin{equation}
\label{e.fromabove-nu}
\frac 1 2 p \cdot \ahom p \le \E[\nu(\cu,p)].
\end{equation}
By \eqref{e.defJ}, this implies
\begin{equation}
\label{e.fromabove-mu}
\frac 1 2 q \cdot \ahom^{-1} q \le -\E[\mu(\cu,q)],
\end{equation}
so that
\begin{equation}
\label{e.fromabove-J}
\E[J(\cu,p,q)] \ge \frac 1 2 p \cdot \ahom p + \frac 1 2 q \cdot \ahom^{-1} q - p\cdot q.
\end{equation}

\smallskip

Due to Lemma~\ref{l.Jconvex}, there exists a unique, deterministic matrix $Q=Q(U)$ such that, for every $p\in\Rd$,
\begin{equation*}
p = \nabla_q \E \left[  J(U,0,Qp) \right] = \E \left[  \nabla_q J(U,0,Qp) \right]. 
\end{equation*}
Indeed, for each fixed $p\in\Rd$, we can consider the minimum of the uniformly convex, quadratic function
\begin{equation*}
q\mapsto \E \left[ J(U,0,q) \right] - p\cdot q. 
\end{equation*}
This defines a linear map from $p$ to the minimum point $q(p)$. The matrix~$Q$ is defined by $q = Qp$. It is easy to check that $Q$ is symmetric, and the upper and lower uniform convexity of $J(U,0,\cdot)$ ensures that $Q$ is positive and in particular invertible. In fact, we have
\begin{equation} 
\label{e.Q-unif}
I_d \leq Q \leq \Lambda I_d. 
\end{equation}
Notice that, for every $\cu \in\mcl C$, we have 
\begin{equation} \label{e.Qbndsa}
Q(\cu) \leq \ahom.
\end{equation}
Indeed, by \eqref{e.fromabove-J},
\begin{align*}
\frac 12 p\cdot Q(\cu) p 
& = \frac12 Q(\cu) p \cdot \E \left[  \nabla_q J(U,0,Q(\cu)p) \right] \\
& = \E \left[ J(U,0,Q(\cu)p) \right] \\
& \geq \frac 12 Q(\cu)p \cdot \ahom^{-1} Q(\cu)p.
\end{align*}
Putting $p:= Q(\cu)^{-1}q$ gives $Q(\cu)^{-1} \geq \ahom^{-1}$, which is equivalent to the claim. For future reference, we notice that since $u(\cdot,U,p,Qp) = u(\cdot,U,p,0) + u(\cdot,U,0,Qp)$ and $-u(\cdot,U,p,0)$ is the minimizer of $\nu(U,p)$, we have that $-u(\cdot,U,p,0) \in \ell_{-p}+H^1_0(U)$ and
\begin{equation*} \label{}
\fint_U \nabla u(x,U,p,0) \,dx = -p
\end{equation*}
and thus by~\eqref{e.Kidentity} that 
\begin{align*} \label{}
\nabla_q \E \left[ J(U,p,Qp) \right] 
& = \nabla_q \E \left[ J(U,0,Qp) \right] +  \nabla_q \E \left[ J(U,p,0) \right] = p - p =0.
\end{align*}

Similarly, we denote by $P=P(U)$ the deterministic, symmetric matrix defined via the relation
\begin{equation*} \label{}
q = \nabla_p\E\left[ J(U,Pq,0) \right].
\end{equation*}
As above, we have the estimates
\begin{equation*} \label{}
\frac1\Lambda I_d \leq P(U) \leq I_d,
\end{equation*}
for every $\cu \in \mathcal C$,
\begin{equation*} \label{}
P(\cu) \leq \ahom^{-1} 
\end{equation*}
and 
\begin{equation} 
\label{e.lookoverhere}
\nabla_q \E \left[ J(U,Pq,q) \right] 
 =0.
\end{equation}

\smallskip 

We next give an estimate for the difference between~$Q(\cu)$ and the homogenized coefficients~$\ahom$ in terms of the expected size of $J(\cu,p,Q(\cu)p)$. 

\begin{lemma}
\label{l.fluxmaps}
There exists $C(d,\Lambda) < \infty$ such that, for any $\cu\in\mcl C$,\begin{equation*}
\left| \ahom - Q(\cu) \right| + \left| \ahom^{-1} - P(\cu)\right| \leq C \sup_{p\in\Rd} |p|^{-2}\, \E\left[ J(\cu,p,Q(\cu)p) \right].
\end{equation*}
\end{lemma}
\begin{proof}
We prove only the estimate for $\left| \ahom - Q(\cu) \right|$. The argument for the bound on $\left| \ahom^{-1} - P(\cu)\right|$ is similar. We drop the dependence on $\cu$ and write $Q=Q(\cu)$ and $J(p,q)= J(\cu,p,q)$. Set 
\begin{equation*} \label{}
\eta := \sup_{p\in\Rd} |p|^{-2}\, \E\left[ J(p,Qp) \right].
\end{equation*}
By~\eqref{e.fromabove-J}, we have, for every $p\in\Rd$,
\begin{equation*}
0 
 \leq \frac12 p\cdot \ahom p + \frac12 Qp \cdot \ahom^{-1} Qp - p\cdot Qp  
 \leq \E \left[ J(p,Qp) \right] 
 \leq \eta |p|^2.
\end{equation*}
It follows by uniform convexity that, for every $p\in\Rd$, 
\begin{equation*}
\left| Qp - \ahom p \right|^2 \leq C\eta |p|^2
\end{equation*}
and thus, by upper uniform convexity of $q\mapsto \E \left[ J(p,q) \right]$ and the fact that this map achieves its minimum at $q=Qp$, we deduce 
\begin{equation*}
\E\left[ J(p,\ahom p) \right] \leq C\eta|p|^2. 
\end{equation*}
Using~\eqref{e.modidentity}, \eqref{e.fromabove-nu} and~\eqref{e.fromabove-mu}, we get 
\begin{equation*}
 p\cdot \ahom p \leq \E\left[ J(p,0) \right] + \E\left[ J(0,\ahom p) \right] \leq  p\cdot \ahom p + C\eta|p|^2. 
\end{equation*}
In view of~\eqref{e.fromabove-nu} and~\eqref{e.fromabove-mu}, this implies that, for every $p,q\in\Rd$, 
\begin{equation*}
\left| \E\left[ J(p,0) \right] - \frac12p\cdot \ahom p \right|  \leq C\eta|p|^2 
\quad \mbox{and} \quad
\left| \E \left[ J(0,q)\right] - \frac12 q \cdot \ahom^{-1} q \right| \leq C\eta|q|^2.
\end{equation*}
From this, the definition of $Q$ and the fact that $J$ is quadratic, we get 
\begin{equation*}
\frac 12 p\cdot Qp 
= \frac12 Qp \cdot \E \left[  \nabla_q J(0,Qp) \right]
 = \E \left[ J(0,Qp) \right] \leq \frac 12 Qp \cdot \ahom^{-1} Qp + C\eta|p|^2.
\end{equation*}
Thus $Q\ahom^{-1} Q \geq Q - C\eta I_d$, that is, $Q(\ahom^{-1} - Q^{-1})Q \ge -C\eta I_d$. By \eqref{e.Q-unif}, this implies $Q \geq \ahom - C\eta I_d$. In view of~\eqref{e.Qbndsa}, the proof is now complete. 
\end{proof}

\section{The bootstrap argument}
\label{s.boostrap.cubes}

This section is devoted to the proof of the following estimate for $J$. 
\begin{proposition}
\label{p.bootstrap.cubes}
For every $\alpha\in \left( 0,\frac1d\right)$, there exist $C(d,\Lambda,\alpha) < \infty$ such that, for every $p\in\Rd$ and $\cu \in \mcl C$,
\begin{equation} \label{e.cubesavage}
\E \left[ J\left( \cu,p,\ahom p \right)  \right]  \leq C|p|^2 \, |\cu|^{-\al}.
\end{equation}
\end{proposition}
An immediate consequence of Proposition~\ref{p.bootstrap.cubes} is
\begin{corollary}
\label{c.bootstrap}
For every $\alpha\in \left( 0,\frac1d\right)$, there exist $C(d,\Lambda,\alpha) < \infty$ such that, for every $p, q \in\Rd$ and $\cu \in \mcl C$,
\begin{equation} \label{e.statbiasests}
\left\{ 
\begin{aligned}
& \left| \E \left[ \nu(\cu,p) \right] - \frac12 p\cdot \ahom p \right| \leq C |p|^2 \, |\cu|^{-\al}, \qquad \mbox{and} \\
& \left|\E \left[ \mu(\cu,q) \right] + \frac12 q\cdot \ahom^{-1}q  \right|
\leq C|q|^2 \, |\cu|^{-\al}.
\end{aligned}
\right. 
\end{equation}
\end{corollary}
\begin{proof}[Proof of Corollary~\ref{c.bootstrap}]
By \eqref{e.modidentity} and Proposition~\ref{p.bootstrap.cubes}, 
\begin{equation} \label{e.cubesavgs2}
\E \left[ \nu(\cu,p) \right] - \E \left[ \mu(\cu,\ahom p) \right] - p\cdot \ahom p \leq C |p|^2 \, |\cu|^{-\al}.
\end{equation}
Combining this with \eqref{e.fromabove-nu} and \eqref{e.fromabove-mu} yields the result.
\end{proof}

\smallskip

We begin the proof of Proposition~\ref{p.bootstrap.cubes} with some reductions. First, observe that it suffices to demonstrate~\eqref{e.cubesavage} for cubes of the form~$\cu_R$, $R \ge 1$, since we can then apply the result to translations of the law $\P$ (recall that we only assume $\Z^d$-stationarity). Second, in view of Lemma~\ref{l.fluxmaps}, in order to prove~\eqref{e.cubesavage}, it suffices to show that, for each $\alpha < \frac1d$, there exists $C(d,\Lambda,\alpha)<\infty$ such that, for every $R\geq 1$ and $p\in\Rd$, 
\begin{equation}
\label{e.Qsavage}
\E \left[ J(\cu_R,p,Q(\cu_R)p) \right] \leq C|p|^2R^{-d\alpha}. 
\end{equation}

\smallskip

The proof of~\eqref{e.Qsavage} is by induction. For each $\alpha\in \left(0,1\right)$ and $K\geq 1$, we let $\mathcal{S}(\alpha,K)$ be the assertion that, for every $R\geq 1$ and $p\in\Rd$, 
\begin{equation*}
\E \left[ J(\cu_R,p,Q(\cu_R)p) \right] \leq K|p|^2 R^{-d\alpha}.
\end{equation*}
In view of Lemma~\ref{l.modulation}, the base case of our bootstrap was proved in~\cite{AS}. 

\begin{proposition}[{\cite[Theorem 3.1]{AS}}]
\label{p.thebasecase}
There exists $\alpha_0(d,\Lambda)>0$ and $K_0(d,\Lambda) < \infty$ such that 
\begin{equation}
\label{e.thebasecase}
\mathcal{S}(\alpha_0,K_0) \quad \mbox{holds.}
\end{equation}
\end{proposition}
The previous proposition was proved in~\cite[Section 3]{AS} by showing that minimizers of $\mu(\cu_n,q)$ for large $n$ are expected to be flat-- that is, close to a deterministic affine function-- and this allows comparison to $\nu(\cu_n,p)$ for an appropriate $p$ (the slope of the affine function). Note that the result was proved with $\ahom$ in place of $Q(\cu_R)$, but since $Q(\cu_R)$ is the minimum of the map $q\mapsto\E\left[ J(\cu_R,p,q)\right]$, we can make this replacement. 

\smallskip

In view of Proposition~\ref{p.thebasecase}, it therefore suffices to show that there exist $\ep(d,\Lambda)>0$ and $C(d,\Lambda) < \infty$ such that, for every $\alpha \in \left[ \alpha_0,\frac1d \right)$ and $K\geq1$,
\begin{equation}
\label{e.indystepcubes}
\mathcal{S}(\alpha,K) \implies \mathcal{S}\left(\alpha+\ep(1-d\alpha),CK^{\frac32}\right).
\end{equation}
Indeed, an iteration of~\eqref{e.indystepcubes}, starting from $\mathcal{S}(\alpha_0,K_0)$, yields~\eqref{e.Qsavage}. 

\smallskip

The rest of this section is focused on the proof of~\eqref{e.indystepcubes}. Throughout, we fix $\alpha \in \left[ \alpha_0,\frac1d\right)$ and $K\geq 1$, and assume that $\mathcal{S}(\alpha,K)$ holds. We fix $R\geq 1$ and $p\in\Rd$. By homogeneity, we may assume that $|p|\leq 1$. We set 
\begin{equation} \label{e.defq}
q:= Q(\cu_R)p,
\end{equation}
where we recall that $Q(\cu_R)$ is defined before Lemma~\ref{l.fluxmaps}. Throughout, 
\begin{equation*}
v:= u\left(\cdot,\cu_R,p,q\right)
\end{equation*}
denotes the maximizer of $J(\cu_R,p,q)$. Notice that Lemma~\ref{l.Jconvexity} gives
\begin{equation} \label{e.Du2}
\left\| \nabla v \right\|_{\underline{L}^2(\cu_R)}^2 \leq 4 J\left( \cu_R, p, q \right). 
\end{equation}
In particular, the induction hypothesis yields
\begin{equation} \label{e.EDu2}
\E \left[ \left\| \nabla v \right\|_{\underline{L}^2(\cu_R)}^2 \right] \leq 4KR^{-d\alpha}.
\end{equation}

\subsection{Rough sketch of the argument for~\texorpdfstring{\eqref{e.indystepcubes}}{(4.5)}}
\label{ss.sketch}
Before giving the complete details, we present an informal summary of the proof of~\eqref{e.indystepcubes}. Choose a mesoscopic scale $r\in (1,R)$ and partition the cube $\cu_R$ into smaller cubes of the form $\cu_r(y)$, $y\in\cu_R$. Observe that, using~\eqref{e.firstvar} in the last step,
\begin{align*}
J(\cu_R,p,q) & = \Ll( \frac R r  \Rr)^{-d} \sum_y \J(v,\cu_r(y),p,q) \\
	& \le \Ll( \frac R r  \Rr)^{-d} \sum_y \innerply{-\a p + q}{\nabla v}{\cu_r(y)} \\
	& = \Ll( \frac R r  \Rr)^{-d} \sum_y \innerply{\a \nabla \tilde u_y}{\nabla v}{\cu_r(y)} ,
\end{align*}
where $\tilde u_y := u(\cu_r(y),p,q)$. If the induction hypothesis $\mcl S(\al,K)$ holds, then, using also~\eqref{e.EDu2} and the Lipschitz estimate, we have 
\begin{equation} \label{e.sketchbnds}
\E[\|\nabla \tilde u_y\|_{\underline{L}^2(\cu_r(y))}^2] \le C r^{-d\al} 
\quad \mbox{and} \quad 
\E[\|\nabla v\|_{\underline{L}^2(\cu_r(y))}^2] \le C R^{-d\al}.
\end{equation}
No improvement on the estimate of $\E[J(\cu_R,p,q)]$ can be obtained from these observations alone, as we have not yet used the mixing properties of the coefficients. 

\smallskip

If~$v$ is sufficiently close (in say $L^2$) to an affine function of slope~$p_y$ in the mesoscopic cube~$\cu_r(y)$, then, since~$\tilde u_y$ is a solution of~\eqref{e.pde}, we may integrate by parts to get
\begin{equation*} \label{}
\innerply{\a \nabla \tilde u_y}{\nabla v}{\cu_r(y)} = \innerply{\a \nabla \tilde u_y}{p_y}{\cu_r(y)} + \mbox{a small error.}
\end{equation*}
Recall that the dual vector $q$ was chosen in~\eqref{e.defq} so that 
\begin{equation*} \label{}
\E\Ll[\fint_{\cu_R} \nabla v(x) \, dx \Rr] = 0,  
\quad \mbox{that is,}\quad
\sum_y \E\Ll[p_y \Rr] = 0.
\end{equation*}
Hence, the sum 
\begin{equation}
\label{e.cov}
\Ll( \frac R r  \Rr)^{-d} \sum_y \E[\innerply{\a \nabla \tilde u_y}{p_y}{\cu_r(y)}]  = \Ll( \frac R r  \Rr)^{-d} \sum_y \E\Ll[p_y \cdot \fint_{\cu_y(r)} \a \nabla \tilde u_y(x) \, dx \Rr]  
\end{equation}
can be seen as a measure of the correlations between the (discrete) random fields $\left(\fint_{\cu_r(y)} \a \nabla \tilde u_y\right)$ and $\left(p_y\right)$. 

\smallskip

The main insight is that the higher regularity theory guarantees that the slopes $p_y$ exist, that the approximation to $v$ has a sufficiently small error, and that $p_y$ \emph{changes very slowly in} $y$. Meanwhile, the spatial averages of the fluxes $\left(\fint_{\cu_r(y)} \a \nabla \tilde u_y\right)$ are \emph{essentially independent as $y$ varies}. If we denote by~$s$ a mesoscale with $r \ll s \ll R$ on which the vectors $p_y$ are approximately constant, then we can expect to bound~\eqref{e.cov} by
\begin{equation} \label{e.CLTscaling0}
R^{-d\al/2} r^{-d\al/2} \Ll(\frac s r \Rr)^{-d/2}.
\end{equation}
Indeed, by~\eqref{e.sketchbnds}, $R^{-d\alpha/2}$ is the expected size of each~$p_y$ and $r^{-d\alpha/2}$ is the expected size of the spatially averaged flux in a mesocube. The factor of~$(s/r)^{-d/2}$ comes from the CLT scaling, as~$(s/r)^d$ is the number of smaller mesocubes of size $r$ in each larger mesocube of size~$s$.

\smallskip

The expression in~\eqref{e.CLTscaling0} can be made smaller than $R^{-d\alpha}$ by choosing the mesoscales $r$ and $s$ appropriately, provided that $\alpha < 1$. This suggests that the correct estimate for $\E \left[ J(\cu_R,p,q) \right]$ should be (up to possible logarithmic factors) $R^{-d}$. The argument we present below does not perform so well, and saturates at $R^{-1}$ due to a boundary layer we have neglected in this sketch. 

\smallskip

\subsection{Declaration of parameters and mesoscales}
We now proceed with the rigorous argument for~\eqref{e.indystepcubes}. 
We take mesoscopic scales $r,s,l \geq 1$ such that
\begin{equation}
\label{e.integerratios}
\frac{R}{l}, \, \frac{l}{3s},\, \frac{s}{r} \in\N
\end{equation}
and 
\begin{equation} \label{e.mesoscales}
r = CR^{1-(m+2) \ep}\,, \quad s = CR^{1-2\ep}\,, \qquad l = CR^{1-\ep}\,, 
\end{equation}
for fixed parameters  $\ep\in \left( 0,\frac1{4m} \right]$ and $m \geq \max\{d,2\}$ to be selected below. Here the constants $C$ in~\eqref{e.mesoscales} are very close to~1 and are chosen so that the constraints~\eqref{e.integerratios} are satisfied. 

\smallskip

The largest mesoscale~$l$ denotes the thickness of a boundary layer we remove from~$\cu_R$ in the first part of the argument; it also serves as a reference scale for the regularity estimates we will apply at the smaller scale~$s$, where we apply Theorem~\ref{t.mesoregularity} to obtain polynomial approximations to~$v$. Finally,~$r$ denotes the size of the smallest mesoscale cells in which we compare~$v$ to local solutions (denoted by~$v_z$ below) chosen to match the polynomial approximations made on the larger scale. Each of these mesoscales, even the smallest, will be chosen to be very close to the macroscale~$R$. 

\smallskip

The degree of the polynomials used in the mesoscopic approximation will be $k\in\N$, also to be chosen. We let the exponent $\delta=\delta(d,\Lambda) >0$ be the minimum of the exponent given in the statement of Theorem~\ref{t.mesoregularity} for the choice $s=1$ and the one given in the statement of Proposition~\ref{p.quenchedEE} for the choices $s=1$, $U=\cu_{1/\sqrt{d}}$ and $\ep=1$. We may assume $\delta \leq 1$. We let $\X$ denote the maximum of the random variables $\X$ and $\mathcal{R}$ from Theorem~\ref{t.mesoregularity} and Proposition~\ref{p.quenchedEE}, respectively, with the same choices of parameters. 

\smallskip

In order to apply Theorem~\ref{t.mesoregularity} from scale $l$ to scale $s$, we must have that
\begin{equation} \label{e.mesoregisOK}
s\geq l^{\frac{k}{k+\delta}}.
\end{equation}
For this, it suffices to impose the restriction
\begin{equation} \label{e.mesoscales0}
\ep \leq \frac{\delta}{2(k+1)}. 
\end{equation}
In the course of the proof, we will also find it convenient to fix the parameters $m,k,\ep$ as follows: 
\begin{equation} \label{e.mesoscales2}
m := \frac{2(1+\al)}{1-\alpha} \vee \frac{2d(1+\al)}{2-d\alpha} \vee d
\end{equation}
and
\begin{equation} \label{e.mesoscales1}
k := \left\lceil \frac{10d(1+\theta) m}{\theta} \right\rceil  \quad \mbox{and} \quad \ep  :=  \frac{\alpha}{2(m+3)} \wedge  \frac{\delta}{2(k+1)}  \wedge \frac{\delta \theta}{20d(1+\theta) m} .
\end{equation}
Here the positive parameter $\theta(d,\Lambda)$ is the exponent related to the local and global Meyers' estimates, which in particular give us that, for every $\psi \in W^{1,2(1+\theta)}(\cu_r(y))$, $w \in \A(\cu_{2r}(y))$ and $\tilde w \in (\psi + H_0^1(\cu_r(y))) \cap \A(\cu_r(y))$, 
\begin{equation} \label{e.meyers1}
\left\| \nabla w \right\|_{\underline{L}^{2(1+\theta)}(\cu_r(y))} \leq C \left\| \nabla w \right\|_{\underline{L}^{2}(\cu_{2r}(y))} 
\end{equation}
and
\begin{equation} \label{e.meyers2}
\left\| \nabla \tilde w \right\|_{\underline{L}^{2(1+\theta)}(\cu_r(y))} \leq C \left\| \nabla \psi \right\|_{\underline{L}^{2(1+\theta)}(\cu_{r}(y))}\,.
\end{equation}
Notice that, since $\alpha \in \left[ \alpha_0,\frac1d\right)$, we have $\ep\geq c(d,\Lambda)>0$. In fact, each of $k$, $m$, and $\ep$ are bounded above and below by constants depending only on $(d,\Lambda)$. This implies that each of the constants $C$ and $c$ in the estimates below will depend only on $(d,\Lambda)$ instead of $(d,\Lambda,k,m,\ep)$.

\smallskip

The improvement in the exponent~$\alpha$ will be the result of a CLT scaling arising due to the fact that the~$v_z$'s inside each~$s$ sized cell are independent of each other. This is seen in the proof in Step~3, below. In the final step of the proof, we will define each of the mesoscales and other parameters in such a way that the error terms encountered in the argument will be at most $CKR^{-d\alpha-\ep(1-d\alpha)}$.

\subsection{Removal of boundary layer}

The improvement in the exponent~$\alpha$ is based on the application of regularity estimates to~$v$. Since we do not have a boundary condition for~$v$, these estimates are inapplicable near the boundary of the macroscopic cube~$\cu_R$. Therefore we must remove a boundary layer, which is accomplished very simply by subadditivity. It is the error we make here that forces our bootstrap argument to halt before~$d\alpha = 1$.

\smallskip

As $R/l$ is an integer, the cubes $\left\{ \cu_l(y) \right\}_{y\in l\Zd \cap \cu_R}$ form a partition of $\cu_R$ (up to a set of Lebesgue measure zero). We let $\cu_R^\circ$ denote the cube obtained from $\cu_R$ after removing a mesoscopic boundary layer of thickness $l$:
\begin{equation*}
\cu_R^\circ := \cu_R \setminus \bigcup\left\{ \cu_l(y)\,:\, y\in l\Zd, \partial \cu_l \cap \partial \cu_R \neq \emptyset \right\}.
\end{equation*}
Notice that 
\begin{equation*}
\frac{\left| \cu_R \setminus \cu_R^\circ \right|}{\left| \cu_R \right|} \leq C\left( \frac Rl \right)^{-1}. 
\end{equation*}
Then by subadditivity, stationarity, Lemma~\ref{l.fluxmaps} and the induction hypothesis, we obtain
\begin{align}
\label{e.boundarylayer}
\lefteqn{
\E \left[ J(\cu_R,p,q) \right]
} \qquad & \\
& \leq \frac{\left| \cu_R^\circ \right|}{\left| \cu_R \right|} \E \left[ J(\cu_R^\circ,p,q) \right]+\frac{\left| \cu_R \setminus \cu_R^\circ \right|}{\left| \cu_R \right|} \E \left[ J(\cu_R \setminus \cu_R^\circ,p,q)\right] \notag \\
& \leq  \frac{\left| \cu_R^\circ \right|}{\left| \cu_R \right|} \E \left[ J(\cu_R^\circ,p,q) \right] + CK\left( \frac{R}{l} \right)^{-1} l^{-d\alpha}. \notag
\end{align}
In the second inequality in the display above, we used the fact that $\cu_R \setminus \cu_R^\circ$ is a union of cubes of the form $y+\cu_l$; the induction hypothesis and Lemma~\ref{l.fluxmaps} imply that 
\begin{equation}
\label{e.fluxmapsrR}
\E\Ll[| Q(\cu_R) - Q(y+\cu_l)| \Rr] \leq CKl^{-d\alpha},
\end{equation}
so that another use of the induction hypothesis gives the desired estimate. 

\smallskip

Throughout the rest of the argument we denote
\begin{equation*}
\mathcal Z:= r\Zd \cap \cu_R^\circ. 
\end{equation*}

\subsection{Local mesoscopic approximations of~\texorpdfstring{$v$}{v}}

The next step in the argument is to introduce local approximations of~$v$ in each mesoscopic cube $\cu_{3r}(y)$, with $y\in \mathcal{Z}$. These are also solutions of~\eqref{e.pde}, that is, members of $\A(\cu_{3r}(y))$. The advantage will be that the approximations, which we denote by $v_y$, serve to localize the dependence on the environment and allow us to exploit the independence assumption. 

\smallskip

We first introduce polynomial approximations of $v$ in the larger mesoscopic cubes of the form $\cu_{3s}(z)$, $z\in s\Zd\cap \cu_R^\circ$. With $k\in\N$ to be defined below, we  select $w_z \in \Ahom_k$ such that 
\begin{equation}
\label{e.defwz}
\left\|  v - w_z \right\|_{\underline{L}^2(\cu_{3s}(z))}  = \inf_{w\in\Ahom_k} \left\|  v - w \right\|_{\underline{L}^2(\cu_{3s}(z))}.
\end{equation}
That is,~$w_z$ is the best approximation to~$v$ in~$L^2(\cu_{3s}(z))$ among $\ahom$-harmonic polynomials of degree at most~$k$. 

\smallskip

We next present some basic estimates concerning the expected size of $w_z$ and the quality of the approximation to~$u$. This is where we use the higher regularity estimates in Theorem~\ref{t.mesoregularity}.

\begin{lemma}
\label{e.basicwz}
Assume that~\eqref{e.mesoscales0} holds. Then there exists $C(d,\Lambda,k)< \infty$ such that, for each $z\in s\Zd \cap \cu_R^\circ$,
\begin{equation} 
\label{e.Ewzapprox}
\E \left[ \left\| v-w_z \right\|_{\underline{L}^2(\cu_{3s}(z))}^2 \right] \\ 
\leq Cl^2\left( \frac{s}{l} \right)^{2k+2}   \E \left[ \left\| \nabla v\right\|^2_{\underline{L}^2(\cu_l(z))}  \right] + CR^{-2d-1},
\end{equation}
\begin{equation}
\label{e.Ewzgradbound}
\E \left[ \left\| \nabla w_z \right\|_{L^\infty(\cu_{3s}(z))}^2 \right] 
\leq C\E \left[ \left\| \nabla v\right\|_{\underline{L}^2(\cu_l(z))}^2 \right] + CR^{-2d-1},
\end{equation}
and, for every $m \in \{2,\ldots,k\}$,
\begin{equation}
\label{e.Ewzbound}
\E \left[ s^{2(m-2)} \left\| \nabla^m w_z \right\|^2_{L^\infty(\cu_{3s}(z))} \right] 
\leq Cl^{-2} \, \E \left[ \left\| \nabla v\right\|^2_{\underline{L}^2(\cu_l(z))} \right] + Cs^{-4}R^{-2d-1}.
\end{equation}
\end{lemma}
\begin{proof}
We break the argument into three steps.

\smallskip

\noindent \emph{Step 1.} The proof of~\eqref{e.Ewzgradbound}. Since $w_z$ is an $\ahom$-harmonic polynomial,  we have
\begin{align} \label{e.detwzgradbnd}
\left\| \nabla w_z \right\|_{L^\infty(\cu_{3s}(z))} 
& \leq \frac Cs \left\| w_z - \left( v \right)_{\cu_{3s}(z)} \right\|_{\underline{L}^2(\cu_{3s}(z))} \\
& \leq \frac Cs  \left\| v - \left( v \right)_{\cu_{3s}(z)} \right\|_{\underline{L}^2(\cu_{3s}(z))} \notag \\
& \leq C  \left\| \nabla v \right\|_{\underline{L}^2(\cu_{3s}(z))},\notag
\end{align}
where, in the above display, we used Lemma~\ref{l.Markov} in the first line,  the optimality of $w_z$ in~\eqref{e.defwz} tested against the constant function $\left( v \right)_{\cu_{3s}(z)}$ and the triangle inequality to get the second line, and finally the Poincar\'e inequality to get the last line. Squaring and taking expectations, we get
\begin{align*}
\lefteqn{
\E \left[ \left\| \nabla w_z \right\|_{L^\infty(\cu_{3s}(z))}^2 \right] 
} \qquad & \\
& \leq C \E\left[ \left\| \nabla v \right\|_{\underline{L}^2(\cu_{3s}(z))}^2 \indc_{\{ T_{z}\X \leq C s\}} \right] + C \E \left[ \left\| \nabla v \right\|_{\underline{L}^2(\cu_{3s}(z))}^2 \indc_{\{ T_{z}\X > C s\}}\right] \\
& \leq C \E \left[ \left\| \nabla v \right\|_{\underline{L}^2(\cu_{l}(z))}^2  \right] +  C \E \left[ \left\| \nabla v \right\|_{\underline{L}^2(\cu_{3s}(z))}^2 \indc_{\{ T_{z}\X > C s\}}\right].
\end{align*}
To bound the second term on the right side, we use~\eqref{e.Du2} and the (deterministic) estimate $J(\cu_R,p,q) \leq C$ to very crudely bound 
\begin{equation*}
\left\| \nabla v\right\|_{\underline{L}^2(\cu_{3s}(z))}^2 \leq C s^{-d} R^{d} \leq CR^d
\end{equation*}
and then combine this with strong integrability of $\X$, which give the following estimate:
\begin{multline}
\label{e.roughbnd}
\E \left[ \indc_{\{ T_{z}\X > s\} } \left\| \nabla v\right\|_{\underline{L}^2(\cu_{3s}(z))}^2  \right]  \\
 \leq CR^{d+2} \P \left[  \X > s \right] 
 \leq C R^{d+2} \exp\left( - c s \right) 
 \leq CR^{-2d-1}.
\end{multline}
In the last line we used $s\geq R^{\frac12}$ from~\eqref{e.mesoscales}. This completes the proof of~\eqref{e.Ewzgradbound}.

\smallskip

\noindent \emph{Step 2.} The proof of~\eqref{e.Ewzapprox}. According to Theorem~\ref{t.mesoregularity} (recall that the assumption~\eqref{e.mesoscales0} implies~\eqref{e.mesoregisOK}), we have
\begin{equation*} \label{}
\left\| v - w_{z} \right\|_{\underline{L}^2(\cu_{3s}(z))} \indc_{\{T_{z}\X \leq C s\}} \\
\leq C \left( \frac{s}{l} \right)^{k+1} l \left\| \nabla v\right\|_{\underline{L}^2(\cu_l(z))} \indc_{\{T_{z}\X \leq C s\}}.
\end{equation*}
Recall that ~$\{ T_y\}_{y\in\Rd}$ is the translation group acting on the probability space~$\Omega$. In the event that $T_z\X$ is too large compared to $s$, we proceed differently. Applying~\eqref{e.defwz} and the Poincar\'e inequality, we get
\begin{align*}
\left\| v - w_{z} \right\|_{\underline{L}^2(\cu_{3s}(z))} \indc_{\{T_{z}\X > C s\}} 
& \leq \left\| v - \left( v \right)_{\cu_{l}(z)} \right\|_{\underline{L}^2(\cu_{3s}(z))} \indc_{\{T_{z}\X > C s\}} \\
& \leq C\left( \frac ls \right)^{\frac d2}\left\| v - \left( v \right)_{\cu_{l}(z)} \right\|_{\underline{L}^2(\cu_{l}(z))} \indc_{\{T_{z}\X > C s\}} \\
& \leq C l \left( \frac ls \right)^{\frac d2} \left\| \nabla v\right\|_{\underline{L}^2(\cu_l(z))} \indc_{\{ T_{z}\X > C s\}}.
\end{align*}
Combining the previous two displays yields
\begin{multline*}
\left\| v - w_{z} \right\|_{\underline{L}^2(\cu_{3s}(z))} \\
\leq C \left( \frac{s}{l} \right)^{k+1} l \left\| \nabla v\right\|_{\underline{L}^2(\cu_l(z))} \indc_{\{T_{z}\X \leq C s\}} + Cl\left( \frac ls \right)^{\frac d2} \left\| \nabla v\right\|_{\underline{L}^2(\cu_l(z))} \indc_{\{ T_{z}\X > C s\}}.
\end{multline*}
Squaring and taking expectations, we obtain
\begin{multline}
\label{e.polyapproxest1}
\E \left[  \left\| v - w_{z} \right\|_{L^\infty(\cu_{3s}(z))}^2 \right] 
\leq C\left( \frac{s}{l} \right)^{2k+2} l^2  \E \left[ \left\| \nabla v\right\|_{\underline{L}^2(\cu_l(z))}^2 \right] 
\\
+ Cl^2 \left( \frac ls \right)^{d}  \E \left[ \indc_{\{ T_{z}\X > s\} } \left\| \nabla v\right\|_{\underline{L}^2(\cu_l(z))}^2  \right].
\end{multline}
To estimate the second term on the right side, we proceed in the same way as the end of Step~1, above, to get 
\begin{align*}
l^2 \left( \frac ls \right)^{d} \E \left[ \indc_{\{ T_{[y]}\X > s\} } \left\| \nabla v\right\|_{\underline{L}^2(\cu_l(z))}^2  \right]  \leq CR^{-2d-1}.
\end{align*}
This completes the proof of~\eqref{e.Ewzapprox}.

\smallskip

\noindent \emph{Step 3.} The proof of~\eqref{e.Ewzbound}. We take $\ell_z \in\Ahom_1$ to be the best affine approximation to~$v$:
\begin{equation*}
\left\|  v - \ell_z \right\|_{\underline{L}^2(\cu_{3s}(z)} 
= \inf_{\ell \in\Ahom_1}\left\|  v - \ell \right\|_{\underline{L}^2(\cu_{3s}(z)}.
\end{equation*}
Using the triangle inequality and~\eqref{e.Ewzapprox} twice, once with $k$ as above and once with $k=1$, we get
\begin{equation*}
\E \left[ \left\| w_z - \ell_z \right\|_{\underline{L}^2(\cu_{3s}(z))}^2\right] 
\leq Cl^2\left( \frac{s}{l} \right)^{4}   \E \left[ \|\nabla v\|_{\underline{L}^2(\cu_l(z))}^2  \right] + CR^{-2d-1}.
\end{equation*}
Since $w_z-\ell_z$ is an $\ahom$-harmonic polynomial, this yields, by Lemma~\ref{l.Markov},
\begin{align*}
\E \left[ \sup_{m\in\{2,\ldots,k\} } s^{2(m-2)} \left\| \nabla^m w_z \right\|^2_{L^\infty(\cu_{3s}(z))} \right] 
& \leq C \E \left[ s^{-4} \left\| w_z - \ell_z \right\|_{L^\infty(\cu_{3s}(z))}^2 \right] \\
& \leq Cl^{-2}  \E \left[ \left\| \nabla v\right\|_{\underline{L}^2(\cu_l(z))}^2 \right] + Cs^{-4}R^{-2d-1},
\end{align*}
which is~\eqref{e.Ewzbound}.
\end{proof}

We next introduce the local mesoscopic approximations to~$v$, which are based on $w_z$. For each $y\in \mathcal{Z}$, we denote by $[y]$ the unique element of $s\Zd \cap \cu_R^\circ$ such that $y \in \cu_s([y])$. For each $y \in \mathcal{Z}$, we let $v_y$ denote the solution of the Dirichlet problem in $\cu_{r}(y)$ with boundary condition $w_{[y]}$. That is, $v_y$ is the unique element of $\mathcal{A}(\cu_{r}(y)) \cap \left( w_{[y]} + H^1_0(\cu_{r}(y)) \right)$.
Note that, for every $\phi\in \A(\cu_r(y))$, 
\begin{equation} \label{e.firstvaruy}
\innerply{\a \left( \nabla v_y - \nabla w_{[y]} \right)}{\nabla \phi}{\cu_r(y)} = 0.
\end{equation}
We next give the estimate for the expected difference between the gradients of~$v$ and~$v_y$, using the previous lemma. 

\begin{lemma}
\label{l.graduuz}
Assume that~\eqref{e.mesoscales0} and~\eqref{e.mesoscales1} hold.
Then there exists a constant $C(d,\Lambda,k) < \infty$ such that
\begin{equation} \label{e.graduuz}
\left| \mathcal Z \right|^{-1} \sum_{y\in \mathcal Z} \E \left[ \left\| \nabla v- \nabla v_y\right\|_{\underline{L}^2(\cu_r(y))}^2  \right] 
\leq  
CKR^{-d\alpha - 4dm\ep}.
\end{equation} 
\end{lemma}
\begin{proof}
Set $\eta_y := v - v_y $ and let 
\begin{equation} \label{e.tfix}
t : = r - r^{1-\tilde \ep}\,, \qquad \tilde \ep := \frac{6d m \ep(1+\theta)}{\theta}\,, 
\end{equation}
where $\theta(d,\Lambda)$ is as in Meyers' estimates~\eqref{e.meyers1} and~\eqref{e.meyers2}. Observe that the choice of $\ep$ in~\eqref{e.mesoscales1} guarantees that $\tilde \ep \leq \frac12$. We first split each of the summands into two pieces as follows:
\begin{equation} \label{e.sum0000}
\int_{\cu_r(y)} \left|\nabla \eta_y(x) \right|^2 \, dx = \int_{\cu_t(y)} \left|\nabla \eta_y(x)\right|^2 \, dx + \int_{y +\cu_r \setminus \cu_t} \left|\nabla \eta_y(x)\right|^2 \, dx\,.
\end{equation}
The first term we will estimate with the aid of the Caccioppoli estimate and Lemma~\ref{e.basicwz}, and the second one using the Meyers' estimates.
Indeed, we have by  the Caccioppoli estimate that
\begin{equation*} 
\frac{1}{|\cu_r|}\int_{\cu_t(y)} \left|\nabla \eta_y(x)\right|^2 \, dx \leq \frac{C}{(r-t)^2} \fint_{\cu_r(y)} \left|\eta_y(x) \right|^2 \, dx \,.
\end{equation*}
To estimate the term on the right, we have by the triangle inequality that
\begin{equation*} 
\fint_{\cu_r(y)} \left|\eta_y(x) \right|^2 \, dx  \leq 2 \left\| v-w_{[y]} \right\|_{\underline{L}^2(\cu_{r}(y))}^2  + 2 \left\| v_y-w_{[y]} \right\|_{\underline{L}^2(\cu_{r}(y))}^2 \,.
\end{equation*}
Now, for the first term we have
\begin{align*} 
\left| \mathcal Z \right|^{-1} \sum_{y\in \mathcal Z}  \left\| v-w_{[y]} \right\|_{\underline{L}^2(\cu_{r}(y))}^2 & = 
\left| \mathcal Z \right|^{-1} \sum_{z \in  s\Z^d \cap \cu_R^\circ}  \left( \frac{s}{r}\right)^d \left\| v-w_{z} \right\|_{\underline{L}^2(\cu_{s}(z))}^2 \\
& \leq C  \left(\frac{s}{R}\right)^d\sum_{z \in  s\Z^d \cap \cu_R^\circ}  \left\| v-w_{z} \right\|_{\underline{L}^2(\cu_{s}(z))}^2 \,,
\end{align*}
and thus we obtain by~\eqref{e.Ewzapprox} that
\begin{align*} 
\lefteqn{ \left| \mathcal Z \right|^{-1} \sum_{y\in \mathcal Z}  \left\| v-w_{[y]} \right\|_{\underline{L}^2(\cu_{r}(y))}^2  } \qquad & \\
& \leq C  \left(\frac{s}{R}\right)^d \sum_{z \in  s\Z^d \cap \cu_R^\circ}   \left( l^2 \left( \frac{s}{l} \right)^{2k+2}   \E \left[ \left\| \nabla v\right\|^2_{\underline{L}^2(\cu_l(z))}  \right] + CR^{-2d-1} \right)\\ 
& \leq C \left( l^2 \left( \frac{s}{l} \right)^{2k+2}   \E \left[  \left\| \nabla v\right\|^2_{\underline{L}^2(\cu_R)}  \right]  + R^{-2d-1}\right)\,.
\end{align*}
On the other hand, using Proposition~\ref{p.quenchedEE} and~\eqref{e.Ewzgradbound} we get
\begin{align*} 
\E\left[ \fint_{\cu_r(y)} \left| v_y(x) - w_{[y]}(x)\right|^2 \, dx \right] & \leq C r^{2-2 \delta} \E \left[ \left\| \nabla w_{[y]} \right\|^2_{L^\infty(\cu_r(y))}  \right]
\\ & \leq C r^{2-2 \delta} \left(\E \left[ \left\| \nabla v\right\|^2_{\underline{L}^2(\cu_l(y))}  \right] + C R^{-2d-1}  \right)\,.
\end{align*}
Summing this over $\mathcal Z$ implies
\begin{equation*} 
\left| \mathcal Z \right|^{-1} \sum_{y\in \mathcal Z}  \E\left[ \fint_{\cu_r(y)} \left| v_y(x) - w_{[y]}(x)\right|^2 \, dx \right] \leq 
 C r^{2-2 \delta} \left( E \left[  \left\| \nabla v\right\|^2_{\underline{L}^2(\cu_R)}  \right]  + R^{-d-1}\right)\,.
\end{equation*}
Connecting above estimates and using the fact, from~\eqref{e.mesoscales1}, we have $\delta \geq \tilde \ep$, we arrive at 
\begin{multline} \label{e.sum1000}
\left| \mathcal Z \right|^{-1} \sum_{y\in \mathcal Z} \E\left[  \fint_{\cu_t(y)} \left|\nabla \eta_y(x)\right|^2 \, dx  \right]  
\\Ê\leq C \left(  \left( \frac{l}{r-t}  \right)^{2}\left( \frac{s}{l} \right)^{2k+2}   + \left( \frac{r^{1-\delta}}{r-t} \right)^2  \right) \E \left[ \left\| \nabla v\right\|^2_{\underline{L}^2(\cu_R)}  \right]  + C R^{-d-1}\,.
\end{multline}
By the choice of $k$ and $t$ in~\eqref{e.mesoscales1} and~\eqref{e.tfix}, respectively, we have
 \begin{equation*} 
 \left( \frac{l}{r-t}  \right)^{2}\left( \frac{s}{l} \right)^{2k+2} = R^{2(m\ep + \tilde \ep - k \ep)} = R^{-2\ep \left( k - \frac{6d m(1+\theta)}{\theta} - m \right)} \leq R^{-4dm\ep}\,,
\end{equation*}
and by the condition for $\ep$ in~\eqref{e.mesoscales1}, 
\begin{equation*} 
\left( \frac{r^{1-\delta}}{r-t} \right)^2 = R^{- 2 \left(1-\ep(m+2)\right) \left( \delta - \frac{6d m(1+\theta) \ep}{\theta} \right)} \leq R^{-4dm\ep}\,.
\end{equation*}
Therefore the induction hypothesis yields by way of~\eqref{e.EDu2} that 
\begin{equation} \label{e.sum1100}
\left| \mathcal Z \right|^{-1} \sum_{y\in \mathcal Z} \E\left[  \fint_{\cu_t(y)} \left|\nabla \eta_y(x)\right|^2 \, dx  \right]   \leq CK R^{-d\alpha -4dm\ep} \,.
\end{equation}

\smallskip

To treat the second term in~\eqref{e.sum0000}, Meyers' estimates~\eqref{e.meyers1} and~\eqref{e.meyers2} provide us, via H\"older's inequality, 
\begin{align*} 
\frac1{|\cu_r|} \int_{y + \cu_r \setminus \cu_t} |\nabla \eta(x)|^2  \, dx
           & \leq \left( \frac{|\cu_r \setminus \cu_t|}{|\cu_r|}\right)^{\frac{\theta}{1+\theta} } \left( \fint_{\cu_r(y)}  |\nabla \eta(x)|^{2+2\theta} \, dx \right)^{\frac{1}{1+\theta}} 
	\\ & \leq C \left(\frac{r-t}{r}\right)^{\frac{ \theta }{1+\theta} } \fint_{\cu_{3r}(y)}  \left( |\nabla v(x)| + |\nabla w_{[y]}(x)| \right)^{2}  \, dx\,.
\end{align*}
Following Step 1 in the proof of Lemma~\ref{e.basicwz} and applying~\eqref{e.Ewzgradbound} once more we obtain
\begin{equation*} 
\E\left[ \frac1{|\cu_r|} \int_{y + \cu_r \setminus \cu_t} |\nabla \eta(x)|^2  \, dx \right] \leq C \left(\frac{r-t}{r}\right)^{\frac{  \theta }{1+\theta} } \left( \E \left[ \left\| \nabla v\right\|_{\underline{L}^2(\cu_l(y))}^2 \right] + CR^{-2d-1} \right)\,.
\end{equation*}
By the choice of $t$ in~\eqref{e.tfix} and $\ep$ in~\eqref{e.mesoscales1}  we have
\begin{equation*} 
\left(\frac{r-t}{r}\right)^{\frac{  \theta }{1+\theta} } = R^{ - \left( 1-\ep(m+2)\right)  \frac{\tilde \ep \theta}{1+\theta}} \leq R^{-4dm\ep }\,.
\end{equation*}
Therefore the summation yields, as before,
\begin{equation} \label{e.sum1010}
\left| \mathcal Z \right|^{-1} \sum_{y\in \mathcal Z}  \E\left[ \frac1{|\cu_r|} \int_{y + \cu_r \setminus \cu_t} |\nabla \eta(x)|^2  \, dx \right] \leq CK R^{-d\alpha -4dm\ep} \,.
\end{equation}
Combining~\eqref{e.sum0000} with~\eqref{e.sum1100} and~\eqref{e.sum1010} yields the statement of the lemma.
\end{proof}

Next we show, using the previous lemma and the uniform convexity of $\J$ in $p$ and $q$ separately, that the expected difference between $\J(v,\cu_r(y),p,q)$ and $\J(v_y,\cu_r(y),p,q)$ is small, after averaging over all $y\in \mathcal{Z}$. 

\begin{lemma}
\label{l.Juuz}
Assume~\eqref{e.mesoscales0} and~\eqref{e.mesoscales1}. There exists $C(d,\Lambda,k) < \infty$ such that
\begin{equation}
\label{e.Juuz}
\left| \mathcal{Z} \right|^{-1}
\sum_{y \in \mathcal{Z}} \E \big[ \left| \J\left( v,\cu_r(y),p,q \right) -  \J\left(v_y,\cu_r(y),p,q \right) \right| \big] 
 \leq C KR^{-d(\alpha +\ep)}. 
\end{equation}
\end{lemma}
\begin{proof}
For convenience, let~$\tilde u_y:= u(\cdot,\cu_r(y),p,q)$ denote the minimizer in the definition of~$J(\cu_r(y),p,q)$. Using~\eqref{e.secondvar} twice and summing, we find that
\begin{multline} \label{}
 \J\left( v,\cu_r(y),p,q\right) -  \J\left(v_y,\cu_r(y),p, q\right) \\
 = \frac12 \innerply{\nabla v - \nabla v_y}{\a\left( 2\nabla \tilde u_y - \nabla v -\nabla v_y \right)}{\cu_r(y)}. 
\end{multline}
Thus
\begin{multline} \label{e.Jbreakdown}
 \left|  \J\left( v,\cu_r(y),p,q\right) -  \J\left(v_y,\cu_r(y),p,q \right) \right|  \\
 \leq C \fint_{\cu_{r}(y)} \left| \nabla v(x) - \nabla v_y(x) \right| \left( \left| \nabla \tilde u_y(x) \right| + \left| \nabla v(x) \right| + \left| \nabla v_y(x) \right| \right)\,dx.
 \end{multline}
 We will estimate the term on the right.
By \eqref{e.fluxmapsrR} (with $l$ replaced by $r$), 
\begin{align*} 
\E\left[ \left\| \nabla \tilde u_y \right\|_{\underline{L}^2(\cu_r(y))}^2\right] 
& \leq C \E \left[ J(\cu_r(y),p,q)\right]   \\
& \leq C \E \left[ J(\cu_r(y),p,Q(\cu_r) p)\right] +CKr^{-d\alpha}\notag \\
& \leq CKr^{-d\alpha}.  \notag
\end{align*}
On the other hand, by the induction hypothesis and~\eqref{e.Ewzgradbound}, we have
\begin{align*} 
\left| \mathcal{Z} \right|^{-1}
\sum_{y \in \mathcal{Z}} \E\left[  \left\| \nabla v_y \right\|_{\underline{L}^2(\cu_r(y))}^2  \right]  & \leq C \left| \mathcal{Z} \right|^{-1}
\sum_{y \in \mathcal{Z}} \E \left[  \left\| \nabla w_{[y]} \right\|_{\underline{L}^2(\cu_r(y))}^2   \right]   \\
& \leq CKR^{-d\alpha}.  \notag
\end{align*}
Lemma~\ref{l.graduuz} further gives
\begin{equation*} 
\left| \mathcal{Z} \right|^{-1}
\sum_{y \in \mathcal{Z}} \E\left[  \left\| \nabla v - \nabla v_y \right\|_{\underline{L}^2(\cu_r(y))}^2  \right] \leq CKR^{-d\alpha}\,.
\end{equation*}
Combining above three displays and using the triangle inequality yield
\begin{equation*} 
\left| \mathcal{Z} \right|^{-1}
\sum_{y \in \mathcal{Z}}  \E\left[  \left\| \nabla \tilde u_y \right\|_{\underline{L}^2(\cu_r(y))}^2  +\left\|\nabla v_y\right\|_{\underline{L}^2(\cu_r(y))}^2 + \left\| \nabla v  \right\|_{\underline{L}^2(\cu_r(y))}^2\right] \leq C K r^{-d\alpha}\,.
\end{equation*}
We now use H\"older's inequality, Lemma~\ref{l.graduuz} and the display above to get
\begin{align*} 
& \left| \mathcal{Z} \right|^{-1} \sum_{y\in \mathcal{Z}} \E \left[ \fint_{\cu_{r}(y)} \left| \nabla v(x) - \nabla v_y(x) \right| \left( \left| \nabla \tilde u_y(x) \right| + \left| \nabla v(x) \right| + \left| \nabla v_y(x) \right| \right)    \,dx \right]   
 \\
& \qquad \leq  C \left( \E \left[ \left| \mathcal{Z} \right|^{-1} \sum_{y\in \mathcal{Z}}\left\| \nabla v - \nabla v_y\right\|_{\underline{L}^2(\cu_r(y))}^2 \right] \right)^{\frac12}  C K^{\frac12 } r^{-d\alpha/2} \\
& \qquad \leq CK R^{-d\alpha/2- 2dm\ep} r^{-d\alpha/2} .
\end{align*}
Now we return to~\eqref{e.Jbreakdown} and use the previous estimate to obtain
\begin{equation*} 
\left| \mathcal{Z} \right|^{-1} \sum_{y \in \mathcal{Z}} \E \big[ \left| \J\left( v,\cu_r(y),p,q \right) -  \J\left(v_y,\cu_r(y),p,q \right) \right| \big] 
\leq CK r^{-d\alpha/2} R^{-d\alpha/2} R^{ - 2dm\ep}.
\end{equation*}
By the choice of $r$ in~\eqref{e.mesoscales}, we have $r^{-d\alpha/2} = CR^{-d\alpha/2} R^{d\alpha (m+2) \ep /2}$, and therefore we get the desired result using~$\alpha \leq 1$ and $m\geq \max\{d,2\}$.
\end{proof}


\subsection{Improving the exponent using independence} 
As in the proof of Lemma~\ref{l.Juuz}, we set $\tilde u_y:= u(\cdot,\cu_{r}(y),p,q)$. Applying~\eqref{e.firstvar} and~\eqref{e.firstvaruy}, we find that, for every $y\in\mathcal{Z}$,
\begin{align} \label{e.bigmagic}
\J\left(v_y,\cu_r(y),p,q\right) 
& \leq \innerply{-\a p + q}{\nabla v_y}{\cu_r(y)} \\
& = \innerply{\a \nabla \tilde u_y}{\nabla v_y}{\cu_r(y)} 
 = \innerply{\a \nabla \tilde u_y}{\nabla w_{[y]}}{\cu_r(y)}. \notag
\end{align}
This together with~\eqref{e.boundarylayer} and Lemma~\ref{l.Juuz} imply that
\begin{multline}
\label{e.Jbreakmeso}
\E\left[ J\left(\cu_{R},p,q\right)\right] \le C\left( \frac Rr \right)^{-d} \E \left[ \sum_{y\in \mathcal{Z}} 
\innerply{\a \nabla \tilde u_y}{\nabla w_{[y]}}{\cu_r(y)}
 \right] \\
+CK\left( \frac{R}{l} \right)^{-1} l^{-d\alpha} + CKR^{-d(\alpha+\ep)}. 
\end{multline}
The sum on the right side of~\eqref{e.Jbreakmeso} allows us to see a CLT-type scaling because the terms are essentially independent for each $y$ inside a single larger mesoscopic cube $\cu_s(z)$. This is the mechanism which improves the exponent~$\alpha$. The precise statement we need is formalized in the following lemma. We remark that this is the only point in the proof that we use the choice $q = Q(\cu_R)p$.

\smallskip

\begin{lemma}
\label{l.bigCLTstep}
Assume~\eqref{e.mesoscales0} and~\eqref{e.mesoscales1} hold.
There exists $C(d,\Lambda,k)< \infty$ such that, for every $z\in s\Zd\cap \cu_R^\circ$,
\begin{multline} \label{e.bigCLTstep}
\E \,\Bigg[  
\sup_{w\in \mathcal{P}_k} 
\Bigg\{ \left(\left\| \nabla w \right\|_{L^\infty(\cu_s(z))} + l\left\| \nabla^2 w \right\|_{L^\infty(\cu_s(z))}   \right)^{-2} \\
\times \left( \sum_{y\in r\Zd \cap \cu_{s}(z)} \innerply{\a \nabla \tilde u_y}{\nabla w}{\cu_r(y)} \right)^2 \Bigg\}
\Bigg]  \\
 \leq C K \left( \frac sr\right)^{2d} r^{-d\alpha} \left(   \left( \frac sr \right)^{-d} +  \left(\frac lr \right)^{-2} 
 + Kr^{-d\alpha} \right).
\end{multline} 
\end{lemma}
\begin{proof}
We fix $z\in s\Zd\cap \cu_R^\circ$ throughout the proof. 

\smallskip

We first prove a bound for each particular $w\in\mathcal{P}_k$, and then put the supremum over~$w$ inside the expectation, using linearity and that~$\mathcal{P}_k$ is a finite dimensional vector space. We begin with an estimate of the expectation of each term in the sum. 

\smallskip

\noindent \emph{Step 1.} We show that, for every $w\in \mathcal{P}_k$ and $y\in r\Zd\cap \cu_s(z)$,
\begin{multline}
\label{e.Efixedw}
\left| \E \left[ \innerply{\a \nabla \tilde u_y}{\nabla w}{\cu_r(y)} \right] \right| \\\leq CK r^{-d\alpha}
\left\| \nabla w \right\|_{L^\infty(\cu_s(z))} + CK^{\frac12} r^{1-\frac{d\alpha}2} \left\| \nabla^2 w \right\|_{L^\infty(\cu_s(z))}  . 
\end{multline}
We first observe that
\begin{multline*} \label{}
\left|  \innerply{\a \nabla \tilde u_y}{\nabla w}{\cu_r(y)} - \nabla w (y) \cdot \fint_{\cu_r(y)} \a(x)\nabla \tilde u_y(x) \,dx \right| \\
\leq C\sup_{x\in \cu_{r}(y)} \left| \nabla w(x) - \nabla w(y) \right| \fint_{\cu_r(x)} \left| \nabla \tilde u_y(x) \right| \,dx.
\end{multline*}
Thus, 
by~\eqref{e.fluxmapsrR} and the induction hypothesis,
\begin{align}
\label{e.Efixedw1}
\lefteqn{
\E \left[ \left|  \innerply{\a \nabla \tilde u_y}{\nabla w}{\cu_r(y)} - \nabla w (y) \cdot \fint_{\cu_r(y)} \a(x)\nabla \tilde u_y(x) \,dx \right| \right] 
} \qquad & \\
& \leq C r  \left\| \nabla^2 w\right\|_{L^\infty(\cu_r(y))}  \, \E \left[\fint_{\cu_r(x)} \left| \nabla \tilde u_y(x) \right|^2 \,dx \right]^{\frac12}  \notag \\
& \leq CK^{\frac12} r^{1-\frac{d\alpha}2}\| \nabla^2 w \|_{L^\infty(\cu_r(y))} . \notag
\end{align}
On the other hand, using Lemma~\ref{l.fluxmaps} and the induction hypothesis again, we get, for every $\cu\in\mathcal C$,
\begin{equation*} \label{}
\left| \ahom^{-1} - P(\cu) \right| \leq CK \left| \cu \right|^{-\alpha},
\end{equation*}
from which we deduce, using also~\eqref{e.fluxmapsrR}, that
\begin{equation*} \label{}
\left| p - P(\cu_r(y)) q \right| \leq CK r^{-d\alpha}.
\end{equation*}
Therefore
\begin{align*} \label{}
\left| \E \left[\fint_{\cu_r(y)} \a(x)\nabla \tilde u_y(x)\,dx \right] \right| 
& = \left|  \E \left[ \nabla_pJ\left( \cu_r(y), p, q \right) \right] \right| \\
& \leq \left|  \E \left[ \nabla_pJ\left( \cu_r(y), P(\cu_r(y)) q, q \right) \right] \right| + CK r^{-d\alpha} \\
& = CK r^{-d\alpha},
\end{align*}
as $\E \left[ \nabla_p J\left( \cu_r(y), P(\cu_r(y)) q, q \right)\right]=0$ by~\eqref{e.lookoverhere}, and hence
\begin{equation} \label{e.Efixedw2}
\left| \E \left[ \nabla w (y) \cdot \fint_{\cu_r(y)} \a(x)\nabla \tilde u_y(x) \,dx \right] \right| \leq CK \left| \nabla w(y) \right| r^{-d\alpha}. 
\end{equation}
We now obtain~\eqref{e.Efixedw} by combining~\eqref{e.Efixedw1} and~\eqref{e.Efixedw2} and the triangle inequality.

\smallskip

\noindent \emph{Step 2.} We show using independence that, for every $w\in \mathcal{P}_k$,
\begin{equation}
\label{e.cltfixedw}
\var \left[  
\sum_{y\in r\Zd \cap \cu_{s}(z)} \innerply{\a \nabla \tilde u_y}{\nabla w}{\cu_r(y)} 
\right]  
 \leq CK  r^{-d\alpha} \left( \frac sr \right)^{d} \left\| \nabla w \right\|_{L^\infty(\cu_s(z))}^{2}.
\end{equation}
We expand the variance by writing
\begin{multline*}
\var \left[   
\sum_{y\in r\Zd \cap \cu_{s}(z)} 
\innerply{\a \nabla \tilde u_y}{\nabla w}{\cu_r(y)} 
 \right]   \\
 =  \sum_{y,y'\in r\Zd \cap \cu_{s}(z)} 
 \cov\left[ \,
  \innerply{\a \nabla \tilde u_y}{\nabla w}{\cu_r(y)}
  , \,
   \innerply{\a \nabla \tilde u_{y'}}{\nabla w}{\cu_r(y')} 
 \right]. 
\end{multline*}
Using independence, the fact that $ \innerply{\a \nabla \tilde u_y}{\nabla w}{\cu_r(y)}$ is $\F(\cu_{r}(y))$--measurable and each cube $\cu_{r+1}(y)$ has nonempty intersection with at most $C$ cubes of the form $\cu_{r+1}(y')$ with $y'\in r\Zd \cap\cu_s(z)$, we obtain
\begin{multline*}
\sum_{y,y'\in r\Zd \cap \cu_{s}(z)}\cov\left[ \,
  \innerply{\a \nabla \tilde u_y}{\nabla w}{\cu_r(y)}
  , \,
   \innerply{\a \nabla \tilde u_{y'}}{\nabla w}{\cu_r(y')} 
 \right] \\
 \leq C\sum_{y\in r\Zd \cap \cu_{s}(z)} 
 \var\left[
 \innerply{\a \nabla \tilde u_y}{\nabla w}{\cu_r(y)}
 \right]. 
\end{multline*}
Finally, we observe that the induction hypothesis gives, for each $y\in r\Zd \cap\cu_s(z)$,
\begin{align*}
 \var\left[
 \innerply{\a \nabla \tilde u_y}{\nabla w}{\cu_r(y)}
 \right]
 & \leq C  \E \left[ \left\| \nabla \tilde u_y \right\|_{\underline{L}^2(\cu_r(y))}^2   \right] \|\nabla w\|_{L^{\infty}(\cu_s(z))}^2 \\
& \leq CK  r^{-d\alpha} \|\nabla w\|_{L^{\infty}(\cu_s(z))}^2.
\end{align*}
This completes the proof of~\eqref{e.cltfixedw}. 

\smallskip

\noindent \emph{Step 3.} We complete the proof. Observe that combining the results of the first two steps gives, for every $w\in \mathcal{P}_k$,
\begin{multline}
\label{e.combineCLTsteps}
\E \left[
\left( \sum_{y\in r\Zd \cap \cu_{s}(z)} 
\innerply{\a \nabla \tilde u_y}{\nabla w}{\cu_r(y)} 
\right)^2
\right]  \\
\leq C K\left( \frac sr\right)^{2d} r^{-d\alpha} \left(  \left( \frac sr \right)^{-d} +  \left( \frac lr \right)^{-2}+ Kr^{-d\alpha} \right)  \\
\times
 \left(\left\| \nabla w \right\|_{L^\infty(\cu_s(z))} + l\left\| \nabla^2 w \right\|_{L^\infty(\cu_s(z))}\right)^{2}.
\end{multline}
To conclude, it remains to smuggle the supremum over $w \in \mathcal{P}_k$ inside the expectation. As $\mathcal{P}_k$ is a finite dimensional vector space with dimension depending only on~$(k,d)$, there exists an integer $N(k,d)\in\N$ and $w_1,\ldots,w_N \in \mathcal{P}_k$ such that $\{ w_1,\ldots,w_N\}$ is a basis for $\mathcal{P}_k$, each $w_j$ satisfies 
\begin{equation*} \label{}
\left\| \nabla w_j \right\|_{L^\infty(\cu_s(z))} + l \left\| \nabla^2 w_j \right\|_{L^\infty(\cu_s(z))}= 1,
\end{equation*}
and for any $w\in \mathcal{P}_k$ expressed in the form
\begin{equation*} \label{}
w = c_1w_1 + \cdots + c_N w_N,
\end{equation*}
we have, for some $C(k,d) \geq 1$,
\begin{equation*} \label{}
\left|c_1\right| + \cdots + \left|c_N\right| \leq C \left( \left\| \nabla w \right\|_{L^\infty(\cu_s(z))} + l \left\| \nabla^2 w \right\|_{L^\infty(\cu_s(z))}  \right). 
\end{equation*}
It follows that 
\begin{multline*} \label{}
\sup_{w\in \mathcal{P}_k} \left(  \left\| \nabla w \right\|_{L^\infty(\cu_s(z))} + l\left\| \nabla^2 w \right\|_{L^\infty(\cu_s(z))}  \right)^{-1} 
\left|  \sum_{y\in r\Zd \cap \cu_{s}(z)} 
\innerply{\a \nabla \tilde u_y}{\nabla w}{\cu_r(y)} \right|  
\\
\leq C \sum_{j=1}^N \left|  \sum_{y\in r\Zd \cap \cu_{s}(z)} 
\innerply{\a \nabla \tilde u_y}{\nabla w_j}{\cu_r(y)} \right|.
\end{multline*}
Squaring and taking the expectation of the previous inequality, using that $N\leq C$ and then applying~\eqref{e.combineCLTsteps}, we get
\begin{multline*} \label{}
\E \left[ 
\sup_{w\in \mathcal{P}_k} 
\left(  \left\| \nabla w \right\|_{L^\infty(\cu_s(z))} + l\left\| \nabla^2 w \right\|_{L^\infty(\cu_s(z))} \right)^{-2}
 \left( \sum_{y\in r\Zd \cap \cu_{s}(z)} 
 \innerply{\a \nabla \tilde u_y}{\nabla w}{\cu_r(y)}  \right)^2
\right] \\
\begin{aligned}
& \leq C \E \left[ 
\left( \sum_{j=1}^N 
\left|  \sum_{y\in r\Zd \cap \cu_{s}(z)}
\innerply{\a \nabla \tilde u_y}{\nabla w_j}{\cu_r(y)} 
\right| 
\right)^2 \right] \\
& \leq C \sum_{j=1}^N \E \left[  \left( \sum_{y\in r\Zd \cap \cu_{s}(z)} 
\innerply{\a \nabla \tilde u_y}{\nabla w_j}{\cu_r(y)}   \right)^2 \right] \\
& \leq CK \left( \frac sr\right)^{2d} r^{-d\alpha} \left(  \left( \frac sr \right)^{-d} +  \left( \frac lr \right)^{-2} + Kr^{-d\alpha} \right).
\end{aligned}
\end{multline*}
This completes the proof. 
\end{proof}

Combining~\eqref{e.bigmagic} and Lemma~\ref{l.bigCLTstep}, we obtain
\begin{multline*}
\label{}
\sum_{y\in\mathcal{Z}} \E \left[ \J\left(v_y,\cu_r(y),p,q\right) \right]  
\le C K^{\frac12} \left( \frac sr\right)^{d}r^{-\frac{d\alpha}2} \left(  \left( \frac sr \right)^{-\frac d2} + \left( \frac lr \right)^{-1} + K^{\frac12}r^{-\frac{d\alpha}2} \right)  \\
  \times 
 \sum_{z\in s\Zd \cap \cu_R^\circ} \E \left[ \left( \left\| \nabla w_z \right\|_{L^\infty(\cu_s(z))}^2 + l^2 \left\| \nabla^2 w_z \right\|_{L^\infty(\cu_s(z))}^2\right) \right]^{\frac12}.
\end{multline*}
By Lemma~\ref{e.basicwz} and \eqref{e.EDu2}, the sum in the line above is bounded by
$$
C \left( \frac Rs\right)^{d} K^{\frac12}R^{-\frac{d\alpha}2},
$$ 
so that
\begin{equation*}
\label{}
\sum_{y\in\mathcal{Z}} \E \left[ \J\left(u_y,\cu_r(y),p,q\right) \right]  \le  CK  \left( \frac Rr\right)^{d} R^{-\frac{d\alpha}2}  r^{- \frac{d\alpha}2} \left( \left( \frac sr \right)^{-\frac d2}+\left( \frac lr \right)^{-1} + K^{\frac12}r^{-\frac{d\alpha}2} \right) .
\end{equation*}
Combining this with~\eqref{e.Jbreakmeso} and using $K\geq 1$, we obtain
\begin{multline*}
\E\left[ J\left(\cu_{R},p,q\right)\right]  \\
\leq 
CK^{\frac32} \left(
R^{- \frac{d\alpha}2} r^{- \frac{d\alpha}2} \left( \left( \frac sr \right)^{-\frac d2} +\left( \frac lr\right)^{-1} \right)
+ R^{- \frac{d\alpha}2} r^{-d\alpha} 
+\left( \frac{R}{l} \right)^{-1} l^{-d\alpha} 
+ R^{-d(\alpha+\ep)}
\right).
\end{multline*}
Using the definitions of the mesoscales,
\begin{equation*}
\left\{ 
\begin{aligned}
& R^{- \frac{d\alpha}2} r^{- \frac{d\alpha}2} \left( \frac sr \right)^{-\frac d2}  = C R^{-d\alpha} R^{-d\ep \left( \frac{m}{2}(1-\alpha)- \alpha \right)} \\
& R^{- \frac{d\alpha}2} r^{- \frac{d\alpha}2} \left( \frac lr \right)^{-1}  =  CR^{-d\alpha} R^{-\ep\left( m+1 - d\alpha(m+2)/2 \right)}  \\
& R^{- \frac{d\alpha}2} r^{-d\alpha} = C R^{-d\alpha} R^{d\ep(m+2)-\frac{d\alpha}2} \\
& \left( \frac{R}{l} \right)^{-1} l^{-d\alpha} = CR^{-d\alpha} R^{\ep(d\alpha-1)}.
\end{aligned}
\right.
\end{equation*}
It is easy to check that that the choices made in~\eqref{e.mesoscales0},~\eqref{e.mesoscales2} and~\eqref{e.mesoscales1} guarantee that the first three terms above are bounded by $R^{-d(\alpha+\ep)}$.

\smallskip

In view of these choices of the parameters, it is clear that the third term (the error in removing the boundary layer) is the limiting one, and we deduce that
\begin{equation*}
\E\left[ J\left(\cu_{R},p,q\right)\right] \leq CK\left( R^{-d\alpha} R^{\ep(d\alpha-1)} + R^{-d(\alpha+\ep)} \right) \leq CK  R^{-d\alpha} R^{\ep(d\alpha-1)}.
\end{equation*}
Thus $\mathcal{S}(\alpha+\ep(1-d\alpha)/d,CK)$ holds, for $\ep\geq c(d,\lambda)$ and $C(d,\Lambda)<\infty$ which, after a redefinition of $\ep$, completes the proof of the claim~\eqref{e.indystepcubes} and therefore the proof of Proposition~\ref{p.bootstrap.cubes}.

\section{Improvement of stochastic integrability by subadditivity}
\label{s.magic}

The aim of this section is to complete the proof of Theorem~\ref{t.conv}, by strengthening the stochastic integrability of Proposition~\ref{p.bootstrap.cubes} (or Corollary~\ref{c.bootstrap}) from $L^1$ to exponential moments. The rough argument is as follows. We decompose $\cu_R$ into subcubes $(\cu_r(y))$. By subadditivity, we can bound the upper fluctuations of $\nu(\cu_R,p)$ by the upper fluctuations of the average over $y$ of $\nu(\cu_r(y),p)$, up to an error controlled by the difference $\E[\nu(\cu_r,p)]-\E[\nu(\cu_R,p)]$, which is small by Corollary~\ref{c.bootstrap}. By independence, the average over $y$ of $\nu(\cu_r(y),p)$ is unlikely to be large. The same argument applied to $-\mu$ gives a control of the lower fluctuations of $\mu$. By duality, we can then control upper and lower fluctuations of both quantities.

We make this idea precise in the following general statement. (Recall our slightly non-standard definition of subadditivity in \eqref{e.subadd}, and that $\mcl C$ is the set of cubes $\cu$ such that $|\cu| \ge 1$.)

\begin{theorem}
Let $\delta < 1$, $R_0 < \infty$, and let $\tmu(\cdot) \le \tnu(\cdot)$ be respectively super- and subadditive quantities, such that for every $R \ge R_0$ and $x \in \R^d$,
$$
\tmu(\cu_R(x)) \mbox{ and } \tnu(\cu_R(x)) \mbox{ are $\mcl F(\cu_{R+R^\delta}(x))$-measurable}.
$$ 
Let $\alpha > 0$ and $\beta \in (0,\alpha) \cap (0,1/2]$. Assume that there exists $c < \infty$ such that for every cube $\cu \in \mcl C$,
\begin{equation}
\label{e.mean-control}
 \E[\tnu(\cu)] - \E[\tmu(\cu)]  \le \frac{c}{|\cu|^{\alpha}}.
\end{equation}
Assume furthermore that there exists $c' < \infty$ such that for every cube $\cu$ satisfying $1 \le |\cu| \le 2^d$ and every $\lambda \in \R$, 
\begin{equation}
\label{e.triv-hyp}
\E[\exp(\lambda \tmu(\cu))] \vee \E[\exp(\lambda \tnu(\cu))] \le c' (1+\lambda^2). 
\end{equation}
Then there exists $C = C(d,\delta, R_0, \alpha,\beta,c,c') < \infty$ such that for every cube $\cu \in \mcl C$ and every $\lambda \in \R$,
\begin{equation}
\label{e.stoch-int}
\log \E[\exp(\lambda |\cu|^\beta(\tnu(\cu) - \E[\tnu(\cu)]))] \le C \lambda^2 
,
\end{equation}
and the same estimate holds with $\tnu$ replaced by $\tmu$.
\label{t.stoch-int-tilde}
\end{theorem}
\begin{remark}
\label{r.improve}
The conclusion of Theorem~\ref{t.stoch-int-tilde} can be strengthened in several directions. First, one can replace $\lambda^2$ by $\lambda^2 \wedge \lambda^{\frac 1 {1-\beta}}$ in the right side of \eqref{e.stoch-int}. Moreover, we believe that \eqref{e.stoch-int} holds with $\alpha = \beta$ for $\alpha < \frac12$ (and with a logarithmic correction when $\alpha = \frac12$), and in fact, the proof given below can be adapted to show this stronger result when $\alpha$ is sufficiently small. On the other hand, even if $\tmu$ and $\tnu$ were additive, one could not improve the conclusion to $\beta > \frac12$, by the central limit theorem.
\end{remark}
The proof of Theorem~\ref{t.stoch-int-tilde} makes use of the following lemma.
\begin{lemma}[parabolicity of log-Laplace]
There exists $\lambda_0 \in (0,1]$ such that the following holds. Let $a \ge 0$ and $X$ be a random variable such that $\E[\exp\left(|X|\right)] < \infty$ and $\E[X] = 0$. If the inequality
\begin{equation}
\label{e.logL}
\log \E[\exp\left( \lambda X\right)] \le a \lambda^2
\end{equation}
holds for every $\lambda$ such that $|\lambda| \in [\lambda_0,1]$, then it holds for every $\lambda \in [-1,1]$. 
\label{l:logL} 
\end{lemma}
\begin{proof}
Since $\E[\exp(|X|)] < \infty$, the function $\Psi:= \lambda \mapsto \log \E[\exp( \lambda X)]$ is infinitely differentiable on $(-1,1)$. Its value and first derivative at $0$ vanish, while its second derivative at $|\lambda| < 1$ is bounded by $\E[X^2 \exp( \lambda X)]$. We choose $\lambda_0 > 0$ sufficiently small that for every $\lambda \in [-\lambda_0,\lambda_0]$,
$$
\forall x \in \R, \quad x^2 \exp\left( \lambda x\right) \le \exp (x) +\exp(-x) .
$$
(By symmetry and monotonicity, it suffices to check the inequality for $\lambda = \lambda_0$.) In particular, for every $\lambda \in [-\lambda_0,\lambda_0]$, we have $$\Psi''(\lambda) \le \E[\exp(X)] + \E[\exp(-X)] \le 2 a.$$ The result then follows by integration.
\end{proof}
\begin{proof}[Proof of Theorem~\ref{t.stoch-int-tilde}] 
For any positive integer $m$ and $\CC \ge 0$, we denote by $\mcl A_m(\CC)$ the assertion that for every cube $\cu$ satisfying $1 \le |\cu| \le 3^{dm}$ and every $\lambda \in \R$,
\begin{equation}
\label{e.stoch-int-defP}
\log \E[\exp(\lambda |\cu|^\beta(\tnu(\cu) - \E[\tnu(\cu)]))] \le \CC \lambda^2,
\end{equation}
and that the same estimate holds with $\tnu$ replaced by $\tmu$. 

\medskip

\noindent \emph{Step 1.} We show that for any given $m_0$, there exists $\CC < \infty$ such that $\mathcal A_{m_0}(\CC)$ holds.
Every cube $\cu$ satisfying $1 \le |\cu| \le 3^{dm_0}$ can be decomposed into a finite number of subcubes of side length between $1$ and $2$. By subadditivity, the hypothesis and H\"older's inequality, 
\begin{equation}
\label{e.si:s11}
\log \E[\exp(\lambda (\tnu(\cu) - \E[\tnu(\cu)]))] \lesssim 1+\lambda^2,
\end{equation}
uniformly over cubes $\cu$ satisfying $1 \le |\cu| \le 3^{dm_0}$ and $\lambda \in \R$. The conclusion for $\tnu$ then follows by Lemma~\ref{l:logL}, and the reasoning for $\tmu$ is identical.

\medskip

\noindent \emph{Step 2.} We show that there exists $\eps > 0$ such that for every $m$ sufficiently large and $\CC \ge 1$,
$$
\mcl A_{m-1}(\CC) \implies \mcl A_m((1+3^{-\eps m})\CC).
$$
Since $\prod_m (1+3^{-\eps m}) < \infty$, this will complete the proof of the proposition.

Assuming $\mcl A_{m-1}(\CC)$, we give ourselves a cube $\cu$ of side length $R$ such that $3^{m-1} < R \le 3^{m}$. For notational convenience, we assume that $\cu = (0,R)^d$. We define a partition of $\cu$ into $3^d$ subcubes of side length $L$, each subcube being surrounded by a layer of smaller cubes of side length $\ell = 3^{\gamma m}$ for some $\gamma \in (\delta, 1)$. (We choose a triadic decomposition of $\cu$ for coherence with the rest of the paper, but a dyadic one would be fine too.) In order for the partition to be well-defined, we ask that $L= (R-2\ell)/3$ be an integer multiple of $\ell$. This requirement can easily be taken care of since $\ell \ll R$, so we will neglect it for clarity. We let $(z_i)_{1 \le i \le 3^d}$ be the centers of the $3^d$ subcubes of side length $L$ such that
$$
\big((0,L) \cup (\ell + L,\ell + 2L) \cup (2\ell + 2L, R) \big)^d = \bigcup_{i = 1}^{3^d} \cu_L(z_i),
$$
and $(z_j')_{1 \le j \le N}$ be the centers of the disjoint subcubes of side length $\ell$ such that
$$
\cu \setminus \Ll( \bigcup_{i = 1}^{3^d} \cu_L(z_i) \Rr) = \bigcup_{j = 1}^N \cu_\ell(z_j') \quad \mbox{up to a set of null measure},
$$
where $N = (R^{d} - (3L)^{d})/\ell^d$. By subadditivity,
$$
\tnu(\cu) \le \sum_{i = 1}^{3^d} \frac{|\cu_{L}|}{|\cu|} \tnu(\cu_L(z_i)) + \sum_{j = 1}^{N} \frac{|\cu_{\ell}|}{|\cu|} \tnu(\cu_\ell(z_j')).
$$
Let $r, s \in (1,\infty)$ be such that $1/r +  1/s = 1$. By H\"older's inequality, for every $\lambda \ge 0$,
\begin{align}
& \log \E\Ll[\exp \Ll(\lambda |\cu|^\beta \Ll(\tnu(\cu) - \E[\tnu(\cu)]\Rr)\Rr)\Rr] \notag \\
\label{e.split-exp1}
& \quad \le \frac{1}{r} \log \E\Ll[\exp \Ll(r \lambda |\cu|^\beta \frac{|\cu_{L}|}{|\cu|}\sum_{i = 1}^{3^d}(\tnu(\cu_L(z_i)) - \E[\tnu(\cu)])\Rr)\Rr] \\
\label{e.split-exp2}
&  \qquad + \frac{1}{s} \log \E\Ll[\exp \Ll(s \lambda |\cu|^\beta \frac{|\cu_{\ell}|}{|\cu|} \sum_{j=1}^N(\tnu(\cu_\ell(z_j')) - \E[\tnu(\cu)])\Rr)\Rr].
\end{align}
We decompose the term in \eqref{e.split-exp1} into
\begin{multline}
\label{e.split-exp1b}
\frac{1}{r} \log \E\Ll[\exp \Ll(r \lambda  \frac{|\cu_{L}|}{|\cu|^{1-\beta}}\sum_{i = 1}^{3^d}(\tnu(\cu_L(z_i)) - \E[\tnu(\cu_L(z_i))])\Rr)\Rr] \\ + \lambda \frac{|\cu_{L}|}{|\cu|^{1-\beta}} \sum_{i = 1}^{3^d} (\E[\tnu(\cu_L(z_i))] - \E[\tnu(\cu)]),
\end{multline}
and likewise, the term in \eqref{e.split-exp2} into
\begin{multline}
\label{e.split-exp2b}
\frac{1}{s} \log \E\Ll[\exp \Ll(s \lambda  \frac{|\cu_{\ell}|}{|\cu|^{1-\beta}} \sum_{j=1}^N(\tnu(\cu_\ell(z_j')) - \E[\tnu(\cu_\ell(z_j'))])\Rr)\Rr] \\
 + \lambda \frac{|\cu_{\ell}|}{|\cu|^{1-\beta}} \sum_{j = 1}^N (\E[\tnu(\cu_\ell(z_j'))] - \E[\tnu(\cu)]).
\end{multline}
It follows from sub-/superadditivity and \eqref{e.mean-control} that
$$
\E[\tnu(\cu)] \ge \lim_{|\tilde \cu| \to \infty} \E[\tnu(\tilde \cu)] =\lim_{|\tilde \cu| \to \infty} \E[\tmu(\tilde \cu)] \ge \E[\tmu(\cu_L(z_i))].
$$
Hence, the second term in \eqref{e.split-exp1b} can be estimated using \eqref{e.mean-control}:
$$
\E[\tnu(\cu_L(z_i))] - \E[\tnu(\cu)] \le  \frac{c}{|\cu_{L}|^\alpha},
$$
and similarly,
$$
\E[\tnu(\cu_\ell(z_j'))] - \E[\tnu(\cu)] \le \frac{c}{|\cu_{\ell}|^\alpha}.
$$
By independence, the first term in \eqref{e.split-exp1b} is equal to
$$
\frac{1}{r} \sum_{i = 1}^{3^d} \log \E\Ll[\exp \Ll(r \lambda  \frac{|\cu_{L}|}{|\cu|^{1-\beta}}(\tnu(\cu_L(z_i)) - \E[\tnu(\cu_{L}(z_i))])\Rr)\Rr],
$$
which by the induction hypothesis is bounded by
\begin{equation}
\label{e.split-exp1c}
\frac{3^d \CC}{r} \Ll(\lambda r \frac{|\cu_{L}|^{1-\beta}}{|\cu|^{1-\beta}}\Rr)^2 
\le \CC r  \lambda^2  
,
\end{equation}
since $\beta \le 1/2$ and $3^d |\cu_{L}| \le |\cu|$. In order to estimate the first term in \eqref{e.split-exp2b}, we split the set $\{z_j', j' \le N\}$ into $3^d$ subsets $\{z_j',j \in Z_1\}$, \ldots, $\{z_j',j \in Z_{3^d}\}$ in such a way that if $j_1 \in Z_{k_1}$ and $j_2 \in Z_{k_2}$ with $k_1 \neq k_2$, then the cubes $\cu_{\ell}(z_{j_1})$ and $\cu_{\ell}(z_{j_2})$ are at distance at least $\ell$ from one another. By H\"older's inequality and independence, we get that the first term in \eqref{e.split-exp2b} is bounded by
$$
\frac{1}{3^d s} \sum_{j = 1}^N \log \E\Ll[\exp \Ll(3^d s \lambda  \frac{|\cu_{\ell}|}{|\cu|^{1-\beta}} (\tnu(\cu_\ell(z_j')) - \E[\tnu(\cu_{\ell}(z_j'))])\Rr)\Rr],
$$
which by the induction hypothesis is bounded by
\begin{equation}
\label{e.split-exp2c}
\frac{\CC N}{3^d s} \Ll( 3^d s \lambda \frac{|\cu_{\ell}|^{1-\beta}}{|\cu|^{1-\beta}} \Rr)^2 
\le \CC 3^d s  \frac{N |\cu_{\ell}|}{|\cu|} \lambda^2.
\end{equation}
To sum up, we have shown that
$$
\log \E\Ll[\exp \Ll(\lambda |\cu|^\beta \Ll(\tnu(\cu) - \E[\tnu(\cu)]\Rr)\Rr)\Rr] \le \CC\Ll(r + 3^d s \frac{N|\cu_{\ell}|}{|\cu|} \Rr)\lambda^2 + \lambda c \frac{|\cu|^{\beta}}{|\cu_{\ell}|^{\alpha}}
$$
(recall that $3^d |\cu_{L}| + N |\cu_{\ell}| = |\cu|$), where $r, s \in (1,\infty)$ such that $1/r + 1/s = 1$ are arbitrary. Recall that $\ell = 3^{\gamma m}$ with $\gamma \in (\delta,1)$. Since $\beta < \alpha$, we can choose $\gamma$ sufficiently close to $1$ that
$$
\frac{|\cu|^\beta}{|\cu_{\ell}|^{\alpha}} = 3^{-m\eps}
$$
for some $\eps  > 0$.
Moreover, there exists a constant $c_d$ such that $N \le c_d 3^{m(1-\gamma)(d-1)}$, hence
$$
\log \E\Ll[\exp \Ll(\lambda |\cu|^\beta \Ll(\tnu(\cu) - \E[\tnu(\cu)]\Rr)\Rr)\Rr] \le \CC\Ll(r + s c_d 3^d 3^{-m(1-\gamma)} \Rr)\lambda^2 + \lambda c\,  3^{-m\eps}.
$$
We can now choose $s = 3^{m(1-\gamma)/2}$ and thus obtain that for every $\lambda \ge 0$,
\begin{multline}
\label{e.boundright}
\log \E\Ll[\exp \Ll(\lambda |\cu|^\beta \Ll(\tnu(\cu) - \E[\tnu(\cu)]\Rr)\Rr)\Rr] \\
\le \CC\Ll((1-3^{-m(1-\gamma)/2})^{-1} + c_d 3^{d-m(1-\gamma)/2} \Rr)\lambda^2 + \lambda c\,  3^{-m\eps}.
\end{multline}
The same reasoning applied to $-\tmu$ shows that for every $\lambda \le 0$,
\begin{multline}
\label{e.boundleft}
\log \E\Ll[\exp \Ll(\lambda |\cu|^\beta \Ll(\tmu(\cu) - \E[\tmu(\cu)]\Rr)\Rr)\Rr] \\
 \le \CC\Ll((1-3^{-m(1-\gamma)/2})^{-1} + c_d 3^{d-m(1-\gamma)/2} \Rr)\lambda^2- \lambda c\,  3^{-m\eps}.
\end{multline}
Moreover, by \eqref{e.mean-control}, we can replace $\E[\tmu(\cu)]$ by $\E[\tnu(\cu)]$ in \eqref{e.boundleft} provided we replace $c$ by $2c$ in the right side of \eqref{e.boundleft}. Since $\tnu \ge \tmu$, we deduce that for every $\lambda \in \R$,
\begin{multline*}
\log \E\Ll[\exp \Ll(\lambda |\cu|^\beta \Ll(\tnu(\cu) - \E[\tnu(\cu)]\Rr)\Rr)\Rr] 
\\ \le \CC\Ll((1-3^{-m(1-\gamma)/2})^{-1} + c_d 3^{d-m(1-\gamma)/2} \Rr)\lambda^2+ 2 |\lambda| c\,  3^{-m\eps}.
\end{multline*}
The result for $\tnu$ follows by Lemma~\ref{l:logL} and similar reasoning applies to~$\tmu$.
\end{proof}

We now complete the proof of Theorem~\ref{t.conv} by combining Proposition~\ref{p.bootstrap.cubes} and Theorem~\ref{t.stoch-int-tilde}.

\begin{proof}[Proof of Theorem~\ref{t.conv}] Let $\al < \al' < 1/d$, $p \in \R^d$ be a unit vector, and $q = \ahom p$. By Proposition~\ref{p.bootstrap.cubes}, there exists $c < \infty$ such that for every $\cu \in \mcl C$,
\begin{equation}
\label{e.p.boot.al}
\E[\nu(\cu,p)] - \E[\mu(\cu,q)] - p\cdot q \le \frac c {|\cu|^{\al'}}.
\end{equation}
We apply Theorem~\ref{t.stoch-int-tilde} with $\tnu = \nu(\, \cdot \,, p)$ and $\tmu = \mu(\, \cdot\, , q) + p\cdot q$. The measurability assumption on $\tnu$ and $\tmu$ and the property that $\tnu \le \tmu$ clearly hold. Assumption~\eqref{e.triv-hyp} is also satisfied, since $\tnu$ and $\tmu$ are bounded. We thus obtain the existence of a constant $C < \infty$ such that for every cube $\cu \in \mcl C$ and $\lambda \in \R$,
$$
\log \E[\exp(\lambda |\cu|^\al(\nu(\cu,p) - \E[\nu(\cu,p)]))] \le C \lambda^2 .
$$
Moreover, the constant $C$ does not depend on the unit vector $p$ (since the same is true of the constant $c$ in \eqref{e.p.boot.al}). By Corollary~\ref{c.bootstrap}, we obtain, for every $\lambda \in \R$,
$$
\log \E\Ll[\exp \Ll(\lambda|\cu|^\alpha \Ll|\nu(\cu,p) - \frac 1 2 p \cdot \ahom p  \Rr| \Rr) \Rr] \le C(1+\lambda^2).
$$
The extension to arbitrary $p \in \R^d$ then follows by homogeneity. The reasoning for $\mu$ is identical.
\end{proof} 

\section{Sublinear growth of the correctors}
\label{s.sublin}

In this section we prove Theorem~\ref{t.correctors}. The proof naturally divides into two main steps: first we introduce a functional inequality, which appears to be new and which we term the \emph{multiscale Poincar\'e} inequality. It converts control of spatial averages of gradients into control over the function itself. Then we show, using Theorem~\ref{t.conv}, that estimates of spatial averages of the gradient of the correctors can be reduced to the convergence of the subadditive energy quantities.

\smallskip

In this section it is convenient to work with triadic cubes, so we change the notation from the rest of the paper: for every $m\in\N$, 
\begin{equation*} \label{}
\cu_m:= \left( -\frac12 3^m , \frac123^m \right)^d. 
\end{equation*}

\subsection{Multiscale Poincar\'e inequality}
Here we present an inequality which gives an estimate of the $H^{-1}$ norm of $\nabla u$ in terms of spatial averages of $\nabla u$ in cubes. This can be seen as a generalization of the usual Poincar\'e inequality giving the bound, for every $u\in H^1(\cu_m)$,
\begin{equation}
\label{e.usualPoincare}
\fint_{\cu_m} \left| u(x) - \left( u \right)_{\cu_m} \right|^2\,dx \leq C(d) 3^{2m} \fint_{\cu_m} \left| \nabla u(x) \right|^2\,dx.
\end{equation}
The sharpness of the scaling of the constant $C3^{2m}$ in the Poincar\'e inequality is, of course, realized by considering an affine function. In the following proposition, we show that this scaling can be improved for functions with gradients having small spatial averages relative to their absolute size: roughly, if the gradient is canceling itself out, then the function has smaller oscillation. 

\begin{proposition}[Multiscale Poincar\'e inequality]
\label{p.multiscalepoincare}
Fix $m\in\N$ and, for each $n\in\N$, $n\leq m$, define $\mcl Z_n:= 3^n\Zd \cap \cu_m$. There exists a constant $C(d)<\infty$ such that, for every $u\in H^1(\cu_m)$,
\begin{multline}
\label{e.multiscalepoincare}
\left\| u - \left( u \right)_{\cu_m} \right\|_{\underline{L}^2(\cu_m)} + \|\nabla u\|_{\underline{H}^{-1}(\cu_m)} \\
\leq C \left\|\nabla u \right\|_{\underline{L}^2(\cu_m)}  
+  C \sum_{n=0}^{m-1} 3^n    \left( \left| \mcl Z_n \right|^{-1} \sum_{y\in \mcl Z_n}\left| \left( \nabla u\right)_{y+\cu_{n} } \right|^2 \right)^{\frac12}.
\end{multline}
\end{proposition}
\begin{proof}
We first prove the estimate for $\|\nabla u\|_{\underline{H}^{-1}(\cu_m)}$ and then deduce the estimate for~$\left\| u - \left( u \right)_{\cu_m} \right\|_{\underline{L}^2(\cu_m)}$  as a simple consequence. Without loss of generality, we may suppose $\left( u \right)_{\cu_m}=0$. 

\smallskip

\noindent \emph{Step 1.} The estimate for $\|\nabla u\|_{\underline{H}^{-1}(\cu_m)}$. Fix $\eta \in H^1(\cu_m;\Rd)$ with $$\fint_{\cu_m}  \left| \nabla \eta(x) \right|^{2}\,dx = 1.$$ We denote by $w \in H^2(\cu_m)$ the unique (up to additive constants) solution of the Neumann problem
\begin{equation*}
\left\{ 
\begin{aligned}
& -\Delta w = -\nabla \cdot \eta + b & \quad \mbox{in} & \ \cu_m, \\
 & \partial_\nu w = 0 &\quad \mbox{on} & \ \partial \cu_m,
\end{aligned} 
\right.
 \end{equation*} 
where $b:= \fint_{\cu_m} \nabla \cdot \eta(x)\,dx$ is chosen to ensure solvability. Then according to~\cite{Gris,AJ}, we have $w\in H^2(\cu_m)$ and
 \begin{equation} \label{e.W2pest}
\fint_{\cu_m} \left| \nabla^2 w(x) \right|^{2}\,dx \leq C \fint_{\cu_m} \left| \nabla \eta(x) \right|^{2}\,dx = C.
\end{equation}
Testing the equation for $w$ with $u$ and using $\left( u \right)_{\cu_m} =0$ yields
\begin{equation}
\label{e.poinc.ibp}
\fint_{\cu_m} \nabla u(x)\cdot \eta(x) \, dx = \fint_{\cu_m} \nabla u(x) \cdot \nabla w(x) \, dx.
\end{equation}
For every $n\in \{ 1,\ldots,m\}$ and $z\in \mcl Z_n$, we have
\begin{multline*}
\int_{z+\cu_{n}} \nabla u(x) \cdot \left( \nabla w(x) - \left( \nabla w\right)_{z+\cu_{n}} \right)\,dx \\
= \sum_{y\in \mcl Z_{n-1} \cap (z+\cu_{n})} \int_{y+\cu_{n-1}} \nabla u(x) \cdot \left( \nabla w(x) - \left( \nabla w\right)_{y+\cu_{n-1}} \right)\,dx \\
+ |\cu_{n-1}|\sum_{y\in \mcl Z_{n-1} \cap (z+\cu_{n})} \left(  \left( \nabla w\right)_{z+\cu_n} -  \left( \nabla w\right)_{y+\cu_{n-1} } \right) \cdot \left( \nabla u \right)_{y+\cu_{n-1}}.
\end{multline*}
By the Poincar\'e inequality,
\begin{equation*}
\sum_{y\in \mcl Z_{n-1} \cap (z+\cu_{n})} \left|  \left( \nabla w\right)_{z+\cu_n} -  \left( \nabla w\right)_{y+\cu_{n-1} } \right|^{2}  \leq C3^{2n}  \fint_{z+\cu_n} \left| \nabla^2 w(x) \right|^{2}\,dx.
\end{equation*}
Therefore, after summing over $z\in \mcl Z_n$ and using H\"older's inequality, we get
\begin{multline*}
\sum_{z\in \mcl Z_n} \int_{z+\cu_{n}} \nabla u(x) \cdot \left( \nabla w(x) - \left( \nabla w\right)_{z+\cu_{n}} \right)\,dx \\
\leq \sum_{y\in \mcl Z_{n-1}} \int_{y+\cu_{n-1}} \nabla u(x) \cdot \left( \nabla w(x) - \left( \nabla w\right)_{y+\cu_{n-1}} \right)\,dx \\
+C3^{n(1+d/2)} \left( \int_{\cu_m} \left| \nabla^2 w(x) \right|^{2}\,dx\right)^{\frac1{2}}\left( \sum_{y\in \mcl Z_{n-1}}\left| \left( \nabla u\right)_{y+\cu_{n-1} } \right|^2 \right)^{\frac12}.
\end{multline*}
Iterating this and using~\eqref{e.W2pest}, 
\begin{multline*}
\int_{\cu_m} \nabla u(x) \cdot \nabla w(x) \,dx \leq \sum_{z\in \mcl Z_0} \int_{z+\cu_0} \nabla u(x) \cdot \left( \nabla w(x) - \left( \nabla w\right)_{z+\cu_{0}} \right)\,dx \\
+ C \left| \cu_m\right|^{\frac12} \sum_{n=0}^{m-1} 3^{n(1+d/2)} \left( \sum_{y\in \mcl Z_{n}}\left| \left( \nabla u\right)_{y+\cu_{n} } \right|^2 \right)^{\frac1{2}}.
\end{multline*}
By the Poincar\'e inequality and \eqref{e.W2pest},
$$
\sum_{z \in \mcl Z_0} \int_{z + \cu_0} |\nabla w(x) - (\nabla w)_{z + \cu_0}|^{2} \, dx \le C \int_{\cu_m} |\nabla \eta(x)|^{2} \, dx = C|\cu_m|.
$$
Thus using H\"older's and Young's inequalities, we obtain
\begin{multline*}
\int_{\cu_m} \nabla u(x) \cdot \nabla w(x) \,dx \\
\leq C|\cu_m|^{\frac1{2}} \left(  \int_{\cu_m} \left| \nabla u(x)\right|^2 \,dx \right)^{\frac12}  
+ C\left| \cu_m\right|^{\frac12}  \sum_{n=0}^{m-1} 3^{n(1+d/2)} \left(\sum_{y\in \mcl Z_{n}}\left| \left( \nabla u\right)_{y+\cu_{n} } \right|^2 \right)^{\frac12}.
\end{multline*}
Using~\eqref{e.poinc.ibp} and rearranging the inequality yields the desired conclusion after taking the supremum over all such~$\eta$. 

\smallskip

\noindent \emph{Step 2.}
The estimate for~$\left\| u  \right\|_{\underline{L}^2(\cu_m)}$. By the representation theorem, there exists $\phi \in L^{2}(\cu_m)$ such that $\fint_{\cu_m} \left| \phi(x) \right|^{2}\,dx =1$ and 
\begin{equation} \label{e.ugetphi}
\left( \fint_{\cu_m} \left| u(x)  \right|^2\,dx \right)^{\frac12} 
 =  \fint_{\cu_m} u(x) \phi(x) \,dx = \fint_{\cu_m} u(x)\left( \phi(x) - \left( \phi \right)_{\cu_m} \right) \,dx . 
\end{equation}
Denote by $w \in H^2(\cu_m)$ the unique (up to additive constants) solution of the Neumann problem
\begin{equation*}
\left\{ 
\begin{aligned}
& -\Delta w = \phi -  \left( \phi \right)_{\cu_m} & \quad \mbox{in} & \ \cu_m, \\
&  \partial_\nu w = 0 &\quad \mbox{on} & \ \partial \cu_m.
\end{aligned} 
\right.
 \end{equation*}
We have that $w\in H^2(\cu_m)$ and 
\begin{equation*}
\fint_{\cu_m} \left| \nabla^2 w(x) \right|^{2} \,dx \leq C\fint_{\cu_m} \left| \phi(x) - \left( \phi \right)_{\cu_m} \right|^{2} \,dx\leq C.
\end{equation*}
Testing the equation for $w$ with $u$ yields that 
\begin{equation*}
\fint_{\cu_m} u(x) \left( \phi(x)- \left( \phi \right)_{\cu_m}\right) \,dx = \fint_{\cu_m} \nabla u(x)\cdot \nabla w(x) \,dx.
\end{equation*}
The conclusion now follows from the previous step and~\eqref{e.ugetphi}.
\end{proof}

\subsection{Proof of Theorem~\ref{t.correctors}}
Throughout this subsection, we let $$v(x,U,p) := -u(x,U,p,0) - p\cdot x.$$ That is, $v(\cdot,U,p) \in H^1_0(U)$ is the minimizer in the definition of $\nu(U,p)$ with the plane $x\mapsto p\cdot x$ subtracted. 

\smallskip

The multiscale Poincar\'e inequality motivates us to prove the sublinearity of the correctors by studying the spatial averages of its gradient in mesoscopic cubes. This is accomplished by a very simple energy comparison argument combined with the Lipschitz estimate, which reduce the needed estimates to the convergence of the subadditive quantities. The statements we need are given in the following two lemmas.

\smallskip

In the first lemma, we compare $\nabla v(\cdot,\cu_m,p)$ in the large cube $\cu_m$ to the gradient of the function obtained by gluing together the functions $v(\cdot,z+\cu_m,p)$ on the mesoscopic subgrid~$\left\{ z+\cu_n\,:\, z\in \mathcal{Z}_n \right\}$ and bound the difference in terms of the cell problem energies in these cubes. 

\begin{lemma}
\label{l.betweenscalesL2}
For every $p\in\Rd$ and $m,n\in\N$ with $m\geq n$,
\begin{multline}
\label{e.betweenscalesL2}
\left| \mcl Z_n \right|^{-1}  \sum_{z\in \mcl Z_n} 
\left\| \nabla v(\cdot,\cu_m,p) - \nabla v(\cdot,z+\cu_n,p) \right\|_{\underline{L}^2(z+\cu_n)}^2
\\
\leq  - \nu(\cu_m,p) + \left| \mcl Z_n \right|^{-1}  \sum_{z\in  \mcl Z_n  } \nu(z+\cu_n,p),
\end{multline}
where $\mcl Z_n:= 3^n\Zd\cap \cu_m$. 
\end{lemma}
\begin{proof}
The lemma is a simple consequence of the uniform convexity of the energy functional and a proof can be found in~\cite[Lemma 2.1]{AS}. For the reader's convenience, we also provide the argument here. Define $V \in H^1_0(\cu_m)$ to be the function obtained by gluing together the functions $v(\cdot,z+\cu_n,p)$ for $z\in 3^n\Zd \cap \cu_m$. In other words, $V\equiv v(\cdot,z+\cu_n,p)$ in $z+\cu_n$, for each $z\in 3^n\Zd \cap \cu_m$. Since 
$x\mapsto p\cdot x + v(x,\cu_m,p)$ is a solution of \eqref{e.pde}, we have 
\begin{equation*} \label{}
\innerply{\a\left( p+ \nabla v(\cdot,\cu_m,p)\right)}{\nabla V}{\cu_m}
 = 0. 
\end{equation*}
Using this and a direct computation, we find that 
\begin{align*} \label{}
\lefteqn{
\frac12 \left\|\nabla v(\cdot,\cu_m,p) - \nabla V \right\|_{\underline{L}^2(\cu_m)}^2
} \qquad & \\
& \leq \frac12 \innerply{ \nabla v(\cdot,\cu_m,p) - \nabla V}{\a\left( \nabla v(\cdot,\cu_m,p) - \nabla V\right) }{\cu_m} \\
& = \frac12 \innerply{ p+ \nabla V}{\a\left( p+ \nabla V\right) }{\cu_m}
 -  \frac12 \innerply{ p+ \nabla v(x,\cu_m,p)}{\a\left( p+ \nabla v(x,\cu_m,p)\right) }{\cu_m} \\
 & = \left( \Ll| \mcl Z_n \Rr|^{-1}  \sum_{z\in \mcl Z_n} \nu(z+\cu_n,p) \right) - \nu(\cu_m,p).
\end{align*}
This completes the proof. 
\end{proof}

We now use the Lipschitz estimate to upgrade the previous estimate. This will give us uniform local control as we blow up the macroscopic cube $\cu_m$ to the whole space and thereby yield information on the stationary correctors. 

\begin{lemma}
\label{l.betweenscalesLinfty}
Let $\X$ be as in the statement of Theorem~\ref{t.mesoregularity}. There exists $C(d,\Lambda) < \infty$ such that, for every $p\in\Rd$, $m,n,k\in\N$ with $m\geq n \geq k$ and $\X \leq 3^k$, we have the estimate
\begin{multline}
\label{e.betweenscalesLinfty}
\left\| \nabla v(\cdot,\cu_m,p) - \nabla v(\cdot,\cu_n,p) \right\|_{\underline{L}^2(\cu_k)} \\
\leq C \sum_{l=n}^{m-1} \left( - \nu(\cu_{l+1},p) + 3^{-d} \sum_{z\in 3^l\Zd \cap \cu_{l+1}} \nu(z+\cu_l,p) \right)^{\frac12}. 
\end{multline}
\end{lemma}
\begin{proof}
Fix $k\in\N$ with $\X \leq 3^k$. By Theorem~\ref{t.mesoregularity} (here we just use the Lipschitz estimate) and the previous lemma, for every $l\in\N$, $l\geq k$, we have
\begin{align*} \label{}
\lefteqn{ 
\left\| \nabla v(\cdot,\cu_{l+1} ,p) - \nabla v(\cdot,\cu_l,p) \right\|_{\underline{L}^2(\cu_k)}^2
} \qquad & \\
& \leq C \left\| \nabla v(\cdot,\cu_{l+1} ,p) - \nabla v(\cdot,\cu_l,p) \right\|_{\underline{L}^2(\cu_l)}^2 \\
&  \leq C \left( - \nu(\cu_{l+1},p) + 3^{-d} \sum_{z\in 3^l\Zd \cap \cu_{l+1}} \nu(z+\cu_l,p) \right).
\end{align*}
Note that the first inequality in the display above is obvious if $l=k$, and is a consequence of Theorem~\ref{t.mesoregularity} for $l\geq k+1$. 
Taking the square root, summing over $l = n,\ldots,m-1$ and using the triangle inequality yields~\eqref{e.betweenscalesLinfty}. 
\end{proof}

We next combine the previous two inequalities with the multiscale Poincar\'e inequality to obtain estimates on the corrector itself in a large macroscopic cube terms of $\nu$ in triadic subcubes. 

\begin{lemma}
\label{l.blackbox}
Let $\X$ be as in the statement of Theorem~\ref{t.mesoregularity}. Then, for every $m\in\N$ with $3^m \geq \X$,
\begin{align}
\label{e.blackbox}
\lefteqn{
\left\| \Phi(\cdot,p) - \left( \Phi(\cdot,p) \right)_{\cu_m} \right\|_{\underline{L}^2(\cu_m)}
 + \left\|\nabla \Phi(\cdot,p)\right\|_{\nH^{-1}(\cu_m)}
} \qquad & \\
& \leq C |p| + C \sum_{n=0}^{m-1} 3^n\left( - \nu(\cu_m,p) + \left| \mcl Z_n \right|^{-1}  \sum_{z\in  \mcl Z_n  } \nu(z+\cu_n,p) \right)^{\frac12} \notag \\
& \qquad + C 3^m \sum_{l=m}^{\infty} \left( - \nu(\cu_{l+1},p) + 3^{-d} \sum_{z\in 3^l\Zd \cap \cu_{l+1}} \nu(z+\cu_l,p) \right)^{\frac12}. \notag
\end{align}
\end{lemma}
\begin{proof}
Letting $m\to \infty$ in~\eqref{e.betweenscalesLinfty} yields that, for every $m\in\N$ with $3^m \geq \X$,
\begin{multline}
\label{e.snapdown}
\left\|\nabla \Phi(\cdot,p) - \nabla v(\cdot,\cu_{m},p) \right\|_{\underline{L}^2(\cu_m)}  \\
\leq C \sum_{l=m}^{\infty} \left( - \nu(\cu_{l+1},p) + 3^{-d} \sum_{z\in 3^l\Zd \cap \cu_{l+1}} \nu(z+\cu_l,p) \right)^{\frac12}. 
\end{multline}
For $n\in\N$ with $n\leq m$, we use the fact that $v(\cdot,z+\cu_n,p) \in H^1_0 (z+\cu_n)$ and use integration by parts, the H\"older inequality and Lemma~\ref{l.betweenscalesL2} to get
\begin{align*}
\lefteqn{
\left| \mcl Z_n \right|^{-1} \sum_{y\in \mcl Z_n}\left| \left( \nabla v(\cdot,\cu_m,p)\right)_{y+\cu_{n} } \right|^2 
} \qquad & \\
& \leq\left| \mcl Z_n \right|^{-1}  \sum_{z\in \mcl Z_n} \left\| \nabla v(\cdot,\cu_m,p) - \nabla v(\cdot,z+\cu_n,p) \right\|_{\underline{L}^2(z+\cu_n)}^2 \\
& \leq C \left( - \nu(\cu_m,p) + \left| \mcl Z_n \right|^{-1}  \sum_{z\in  \mcl Z_n  } \nu(z+\cu_n,p) \right).
\end{align*}
Applying Proposition~\ref{p.multiscalepoincare}, we obtain
\begin{align*}
\lefteqn{
\left\| v(\cdot,\cu_m,p)  \right\|_{\underline{L}^2(\cu_m)} + \left\|\nabla v(\cdot,\cu_m,p) \right\|_{\nH^{-1}(\cu_m)}
} \quad & \\
& \leq C\left\| \nabla v(\cdot,\cu_m,p) \right\|_{\underline{L}^2(\cu_m)}
+  C \sum_{n=0}^{m-1} 3^n  \left( \left| \mcl Z_n \right|^{-1} \sum_{y\in \mcl Z_n}\left| \left( \nabla v(\cdot,\cu_m,p)\right)_{y+\cu_{n} } \right|^2 \right)^{\frac12} \notag \\
& \leq C |p| + C \sum_{n=0}^{m-1} 3^n\left( - \nu(\cu_m,p) + \left| \mcl Z_n \right|^{-1}  \sum_{z\in  \mcl Z_n  } \nu(z+\cu_n,p) \right)^{\frac12}.
\end{align*}
By the previous display,~\eqref{e.snapdown}, the triangle inequality and the fact that the $\underline{L}^2$ norm is stronger than the $\nH^{-1}$ norm, we get
\begin{align*} \label{}
\left\|\nabla \Phi(\cdot,p)\right\|_{\nH^{-1}(\cu_m)}
& \leq \left\|\nabla v(\cdot,\cu_m,p) \right\|_{\nH^{-1}(\cu_m)}
 + \left\|\nabla \Phi(\cdot,p) - \nabla v(\cdot,\cu_m,p) \right\|_{\nH^{-1}(\cu_m)} \\
& \leq \left\|\nabla v(\cdot,\cu_m,p) \right\|_{\nH^{-1}(\cu_m)} 
+ C\left\|\nabla \Phi(\cdot,p) - \nabla v(\cdot,\cu_m,p) \right\|_{\underline{L}^2(\cu_m)} \\
& \leq C |p| + C \sum_{n=0}^{m-1} 3^n\left( - \nu(\cu_m,p) + \left| \mcl Z_n \right|^{-1}  \sum_{z\in  \mcl Z_n  } \nu(z+\cu_n,p) \right)^{\frac12} \\
& \qquad + C 3^m \sum_{l=m}^{\infty} \left( - \nu(\cu_{l+1},p) + 3^{-d} \sum_{z\in 3^l\Zd \cap \cu_{l+1}} \nu(z+\cu_l,p) \right)^{\frac12}.
\end{align*}
This is the desired estimate for~$\left\|\nabla \Phi(\cdot,p)\right\|_{\nH^{-1}(\cu_m)}$. To complete the proof of~\eqref{e.blackbox} and obtain the estimate for~$\left\| \Phi(x,p) - \left( \Phi(\cdot,p) \right)_{\cu_m}  \right\|_{\underline{L}^2(\cu_m)}$, we argue similarly, but there is an extra step in which we use the Poincar\'e inequality:
\begin{align*}
\lefteqn{
\left\| \Phi(x,p) - \left( \Phi(\cdot,p) \right)_{\cu_m}  \right\|_{\underline{L}^2(\cu_m)} 
} \  & \\
& \leq 2\left\| v(\cdot,\cu_m,p)  \right\|_{\underline{L}^2(\cu_m)} + \left\| \Phi(x,p) - v(\cdot,\cu_m,p) - \left( \Phi(\cdot,p) - v(\cdot,\cu_m,p) \right)_{\cu_m}   \right\|_{\underline{L}^2(\cu_m)} \\
& \leq C\left\| v(\cdot,\cu_m,p)  \right\|_{\underline{L}^2(\cu_m)}
+ C3^m \left( \fint_{\cu_{m}} \left| \nabla \Phi(x,p) - \nabla v(x,\cu_{m},p) \right|^2\,dx \right)^{\frac12} \\
& \leq C |p| + C \sum_{n=0}^{m-1} 3^n\left( - \nu(\cu_m,p) + \left| \mcl Z_n \right|^{-1}  \sum_{z\in  \mcl Z_n  } \nu(z+\cu_n,p) \right)^{\frac12} \\
& \qquad + C 3^m \sum_{l=m}^{\infty} \left( - \nu(\cu_{l+1},p) + 3^{-d} \sum_{z\in 3^l\Zd \cap \cu_{l+1}} \nu(z+\cu_l,p) \right)^{\frac12}. \qedhere
\end{align*}
\end{proof}

We are now ready to complete the proof of Theorem~\ref{t.correctors}. 

\begin{proof}[{Proof of Theorem~\ref{t.correctors}}]
By standard comparisons, it suffices to establish the result with $B_R$ replaced by $\cu_m$, $m \in \N$. We may also assume $|p| \le 1$ by homogeneity. Fix $\beta \in \left( 0,\frac12\right)$ and define, for each $n\in\N$,
$$
s_n := \Ll\{
\begin{aligned}
3^{-(m-n)}\left( - \nu(\cu_m,p) + \left| \mcl Z_n \right|^{-1}  \sum_{z\in  \mcl Z_n  } \nu(z+\cu_n,p) \right)^{\frac12}  & \quad \mbox{if } n \le m-1, \\
\left( - \nu(\cu_{n+1},p) + 3^{-d} \sum_{z\in 3^n\Zd \cap \cu_{n+1}} \nu(z+\cu_n,p) \right)^{\frac12} & \quad \mbox{if } n \ge m.
\end{aligned}
\Rr. 
$$
By Lemma~\ref{l.blackbox}, the theorem is proved if we can show that 
$$
\log \E \Ll[ \exp \Ll( \lambda 3^{2m\beta} \Ll(\sum_{n = 0}^\infty s_n\Rr)^2 \Rr)  \Rr] \le C(1+\lambda^2).
$$
We let $\beta' \in (\beta,1/2)$, and apply Jensen's inequality with respect to the measure $\sum_{n = 0}^{\infty} 3^{-n(\beta'-\beta)} \delta_n$
and the convex function $x \mapsto \exp( x^2)$ to get
$$
\exp \Ll( \lambda 3^{2m\beta} \Ll(\sum_{n = 0}^\infty s_n\Rr)^2 \Rr) \le C \sum_{n = 0}^{\infty} 3^{-n(\beta'-\beta)} \exp \Ll( \lambda 3^{2m\beta + 2n(\beta'-\beta)} s_n^2 \Rr).
$$
We analyse first the sum over $n \ge m$. 
By Theorem~\ref{t.conv} with $\al = 2\beta'/d$, 
\begin{align*}
\sum_{n = m}^{\infty} 3^{-n(\beta'-\beta)} \E\Ll[\exp \Ll( \lambda 3^{2m\beta + 2n(\beta'-\beta)} s_n^2 \Rr) \Rr] & \le 
\sum_{n = m}^\infty 3^{-n(\beta'-\beta)} \exp\Ll[C\Ll(1+3^{-2\beta(n-m)} \lambda^2 \Rr)\Rr] \\
& \le \exp\Ll[C(1+\lambda^2)\Rr].
\end{align*}
For the sum over $n < m$, we use Jensen's inequality and Theorem~\ref{t.conv} to get
\begin{equation*} \label{}
\log \E \Ll[ \exp \Ll(  \lambda 3^{2n\beta'} \Ll|\frac 1 2 p \cdot \ahom p - \left| \mcl Z_n \right|^{-1}  \sum_{z\in  \mcl Z_n  } \nu(z+\cu_n,p) \Rr| \Rr) \Rr] \le C(1+\lambda^2).
\end{equation*}
Therefore
\begin{multline*}
\sum_{n = 0}^{m-1} 3^{-n(\beta'-\beta)} \E\Ll[\exp \Ll( \lambda 3^{2m\beta + 2n(\beta'-\beta)} s_n^2 \Rr) \Rr] \\
  \leq \sum_{n = 0}^{m-1} 3^{-n(\beta'-\beta)} \exp \Ll[ C \Ll(1+3^{-2(1-\beta)(m-n)} \lambda^2\Rr) \Rr] 
 \leq \exp \Ll[ C \Ll(1+\lambda^2\Rr) \Rr].
\end{multline*}
This completes the proof. 
\end{proof}

\noindent {\bf Acknowledgments.} The second author was supported by the Academy of Finland project \#258000.

\small
\bibliographystyle{plain}
\bibliography{cubes}

\end{document}